\documentclass{article}
\usepackage{amsthm}
\usepackage{amssymb}
\usepackage{amsmath}
\usepackage{epsfig}
\usepackage{colonequals}
\usepackage{enumerate}
\usepackage[dvips]{color}

\newcommand{\QQ}{{\cal Q}}
\newcommand{\CC}{{\cal C}}
\newcommand{\PP}{{\cal P}}

\newtheorem{theorem}{Theorem}[section]
\newtheorem{corollary}[theorem]{Corollary}
\newtheorem{lemma}[theorem]{Lemma}
\newtheorem{observation}[theorem]{Observation}

\newcount\claimno
\def\newclaim#1#2{
   \global\advance\claimno by 1\relax
   \bigskip\noindent\rlap{\rm(\the\claimno)}\ignorespaces
   \global\expandafter\edef\csname CLAIMLABEL#1\endcsname{\the\claimno}\relax
   \hangindent=33pt\hskip30pt{\sl#2}\bigskip}
\def\mylabel#1{{\label{#1}}}
\def\junk#1{}

\begin{document}
\title{Three-coloring triangle-free graphs on surfaces VI.  $3$-colorability of quadrangulations}
\author{%
     Zden\v{e}k Dvo\v{r}\'ak\thanks{Computer Science Institute (CSI) of Charles University,
           Malostransk{\'e} n{\'a}m{\v e}st{\'\i} 25, 118 00 Prague, 
           Czech Republic. E-mail: {\tt rakdver@iuuk.mff.cuni.cz}.
	   Supported by project GA14-19503S (Graph
	    coloring and structure) of Czech Science Foundation and
	    by project LH12095 (New combinatorial algorithms - decompositions, parameterization, efficient solutions) of Czech Ministry of Education.}
 \and
     Daniel Kr{\'a}l'\thanks{Faculty of Informatics,
            Masaryk University, Botanick\'a 68A, 602 00 Brno, Czech Republic, and
            Mathematics Institute, DIMAP and Department of Computer Science, University
            of Warwick, Coventry CV4 7AL, UK. E-mail: {\tt dkral@fi.muni.cz}. The
            previous affiliation: Computer Science Institute (CSI) of Charles University.}
 \and
        Robin Thomas\thanks{School of Mathematics, 
        Georgia Institute of Technology, Atlanta, GA 30332. 
        E-mail: {\tt thomas@math.gatech.edu}.
        Partially supported by NSF Grants No.~DMS-0739366 and DMS-1202640.}
}
\date{August 2020}
\maketitle
\begin{abstract}
We give a linear-time algorithm to decide $3$-colorability (and find a $3$-coloring, if it exists)
of quadrangulations of a fixed surface.  The algorithm also allows to prescribe the coloring for a bounded number of vertices.
\end{abstract}

\section{Introduction}

This paper is a part of a series aimed at studying the $3$-colorability
of graphs (possibly with parallel edges, but no loops) on a fixed surface that are either triangle-free, or have their
triangles restricted in some way.  All colorings in the paper are proper, i.e., adjacent vertices
have different colors.  Embeddability in a surface is not a sufficient restriction by itself, as
$3$-colorability of planar graphs is NP-complete~\cite{garey1979computers}.  Restricting the triangles
is natural in the light of the well-known theorem of Gr\"otzsch stating that every planar triangle-free graph is $3$-colorable.

Quadrangulations of a surface represent an important special case of the problem, as evidenced by the following theorem of Gimbel and Thomassen~\cite{gimbel}.

\begin{theorem}\mylabel{thm:gimtho}
A triangle-free graph embedded in the projective plane is $3$-colorable if and only
if it has no subgraph isomorphic to a non-bipartite quadrangulation of the projective plane.
\end{theorem}

In a previous paper of this series~\cite{trfree4}, we showed how Theorem~\ref{thm:gimtho}
can be generalized to other surfaces.  A graph $G$ is \emph{$4$-critical} if its chromatic number is $4$
and the chromatic number of every proper subgraph of $G$ is at most $3$.

\begin{theorem}[\cite{trfree4}, Theorem 1.3]\mylabel{thm:ctvrty}
There exists a constant $\kappa$ with the following property.  Let $G$ be a graph embedded in
a surface of Euler genus $g$.  Let $t$ be the number of triangles in $G$ and let $c$ be the number of
$4$-cycles in $G$ that do not bound a $2$-cell face.  If $G$ is $4$-critical, then
$$\sum_{\text{$f$ face of $G$}} (|f|-4)\le \kappa(g+t+c-1).$$
\end{theorem}

This result forms a basis for an algorithm to decide the $3$-colorability of a triangle-free graph $G$ embedded in a surface of Euler genus $g$.
Let us for simplicity assume that all $4$-cycles in $G$ bound $2$-cell faces.  By Theorem~\ref{thm:ctvrty},
it suffices to enumerate all subgraphs $H$ of $G$ such that 
\begin{equation}\label{eq-shf}
\sum_{\text{$h$ face of $H$}} (|h|-4)\le \kappa (g-1)
\end{equation}
and test whether they are $3$-colorable. There are $O\bigl(|V(G)|^{5\kappa (g-1)}\bigr)$ subgraphs of $G$ satisfying (\ref{eq-shf})---enumerate
the boundaries of the faces of $H$ of length greater than $4$, delete the vertices and edges of $G$ drawn inside these faces, and test whether
all other faces have length $4$.
Hence, if we can test the $3$-colorability of such a graph $H$ in a polynomial time, we obtain a polynomial-time algorithm to test $3$-colorability of $G$.
Let us remark that the exponent of the polynomial bounding the complexity of this algorithm depends on $g$; in~\cite{trfree7}, we give
a more involved linear-time algorithm.

A graph $H$ satisfying (\ref{eq-shf}) contains at most $5\kappa (g-1)$ vertices incident with faces of length greater than $4$.  Therefore, we can
try all possible $3$-colorings of these vertices and test whether they extend to a $3$-coloring of $H$.  In this paper, we develop
a linear-time algorithm to perform this last step.

Let us now state the result more precisely, starting with a number of definitions.  A \emph{surface} is a compact connected $2$-manifold with a (possibly null) boundary.  Each component of the boundary
is homeomorphic to a circle, and we call it a \emph{cuff}.  For non-negative integers $a$, $b$ and $c$,
let $\Sigma(a,b,c)$ denote the surface obtained from the sphere by adding $a$ handles, $b$ crosscaps and
removing the interiors of $c$ pairwise disjoint closed discs.  The classification theorem of surfaces (see e.g.~\cite{gallier2013guide}) shows that
every surface is homeomorphic to $\Sigma(a,b,c)$ for some choice of $a$, $b$ and $c$.
The \emph{Euler genus} $g(\Sigma)$ of a surface $\Sigma$ homeomorphic to $\Sigma(a,b,c)$ is defined as $2a+b$.

Consider a graph $G$ embedded in the surface $\Sigma$; when useful, we identify $G$ with the topological
space consisting of the points corresponding to the vertices of $G$ and the simple curves corresponding
to the edges of $G$.  A \emph{face} $f$ of $G$ is a maximal connected subset of $\Sigma-G$.
By the \emph{length} $|f|$ of $f$, we mean the sum of the lengths of the boundary walks of $f$ (in particular, if an edge
appears twice in the boundary walks, it contributes $2$ to $|f|$).
A face $f$ is \emph{closed $2$-cell} if it is homeomorphic to an open disk and its boundary forms a cycle in $G$.
A graph $H$ is a \emph{quadrangulation} of a surface $\Sigma$ if all faces of $H$ are closed $2$-cell and have length $4$.
This condition implies the boundary of $\Sigma$ is formed by a set of pairwise vertex-disjoint cycles in $H$,
called the \emph{boundary cycles} of $H$.
A vertex of $G$ contained in the boundary of $\Sigma$ is called a \emph{boundary vertex}.

Finally, we are ready to formulate our main algorithmic result.

\begin{theorem}\label{thm-mainalg}
For every surface $\Sigma$ and integer $k$, there exists a linear-time algorithm with input
\begin{itemize}
\item $G$: a quadrangulation of $\Sigma$ with at most $k$ boundary vertices, and
\item $\psi$: a function from boundary vertices of $G$ to $\{1,2,3\}$,
\end{itemize}
which decides whether there exists a $3$-coloring $\varphi$ of $G$ such that $\varphi(v)=\psi(v)$ for every boundary vertex $v$ of $G$.
In the affirmative case, the algorithm also outputs such a coloring $\varphi$.
\end{theorem}

The first step towards this algorithm was made by Hutchinson~\cite{locplanq}, who proved that for every orientable surface $\Sigma$, there
exists $c$ such that every graph embedded in $\Sigma$ with all faces of even length and without non-contractible cycles of length at most $c$ is $3$-colorable.
One could hope for a similar result in the setting of Theorem~\ref{thm-mainalg}; however, there turns out to be one more obstruction to the existence
of a $3$-coloring, based on \emph{winding number}.

Let $G$ be graph, let $\psi$ be a coloring of $G$ by colors $\{1,2,3\}$, and let $Q=v_1v_2\ldots v_kv_1$ be a directed closed walk in $G$.
One can view $\psi$ as mapping $Q$ to a closed walk in a triangle $T$, and the winding number of $\psi$ on $Q$ is then the number of times this walk goes
around $T$ in a fixed direction.  More precisely, for $uv\in E(G)$, let $\delta_\psi(u,v)=1$ if $\psi(v)-\psi(u)\in\{1,-2\}$, and
$\delta_\psi(u,v)=-1$ otherwise.  For a walk $W=u_1u_2\ldots u_m$, let $\delta_\psi(W)=\sum_{i=1}^{m-1} \delta_\psi(u_i,u_{i+1})$.
The winding number $\omega_\psi(Q)$ of $\psi$ on $Q$ is defined as $\delta_\psi(Q)/3$.

Suppose that $G$ is embeded in an orientable surface $\Sigma$ so that every face of $G$ is closed $2$-cell.  Let $\CC$ be the set consisting of all facial
and boundary cycles of $G$.  Since $\Sigma$ is orientable, we can orient the edges of each cycle of $\CC$
in the clockwise direction around the corresponding face or cuff.  Hence, every edge $e\in G$ is oriented in opposite directions in the two cycles of $\CC$
containing $e$.  For any $3$-coloring $\psi$ of $G$, consider the sum
of its winding numbers on the cycles in $\CC$ in this orientation.  Since each edge appears in opposite orientations
in two cycles of $\CC$, their contributions cancel out.

\begin{observation}\label{obs-win0}
Let $G$ be a graph embedded in an orientable surface $\Sigma$ so that every face of $G$ is closed $2$-cell.
Let $Q_1$, \ldots, $Q_m$ be the cycles of $\CC$ viewed as closed walks in the clockwise orientation.  If $\psi$ is a $3$-coloring of $G$, then
$$\sum_{i=1}^m \omega_\psi(Q_i)=0.$$
\end{observation}

As the winding number of any $3$-coloring on a $4$-cycle is $0$, we obtain the following constraint on $3$-colorings of quadrangulations.

\begin{corollary}\label{cor-win0}
Let $G$ be a quadrangulation of an orientable surface $\Sigma$.  Let $B_1$, \ldots, $B_k$ be the boundary cycles of $G$ in the clockwise orientation.  If $\psi$ is
a $3$-coloring of $G$, then
$$\sum_{i=1}^k \omega_\psi(B_i)=0.$$
\end{corollary}
Therefore, when generalizing Hutchinson's result~\cite{locplanq} to graphs with precolored boundary cycles, for orientable surfaces we need to require that the
sum of the winding numbers of boundary cycles in their prescribed coloring is~$0$.

For non-orientable surfaces, the situation is a bit more complicated, as we cannot consistently orient all cycles of $\CC$.  Suppose that $G$ is a quadrangulation of
a non-orientable surface, and let us fix directed cycles $B_1$, \ldots, $B_k$
tracing the cuffs of~$G$.  For each facial cycle, choose an orientation arbitrarily.
Let $D$ denote the directed graph with vertex set $V(G)$ and with $(u,v)$ being an edge of $D$ if and only if $uv$ is an edge of $G$ oriented
towards $v$ in both cycles of $\CC$ that contain it.  Let $p(G,B_1,\ldots, B_k)=2|E(D)|\bmod 4$.
Note that $p(G,B_1,\ldots, B_k)$ is independent on the choice of the orientations of the $4$-faces, since reversing the orientation
of a $4$-face with $d$ edges belonging to $D$ changes $2|E(D)|$ by $2(4-2d)\equiv 0\pmod{4}$.

Consider the sum of winding numbers of a $3$-coloring $\psi$ of $G$ on cycles in~$\CC$.
As before, the contributions of all other edges of $G$ that do not belong to $D$ cancel out, and
since $G$ is a quadrangulation, the winding number on any non-boundary cycle in $\CC$ is $0$.
Hence,
$$\sum_{i=1}^k \delta_\psi(B_i)=2\sum_{(u,v)\in E(D)} \delta_\psi(u,v).$$
Since $\delta_\psi(u,v)=\pm 1$ for every $uv\in E(D)$,
$$2\sum_{(u,v)\in E(D)} \delta_\psi(u,v)\equiv 2|E(D)|\equiv -2|E(D)|\equiv -p(G,B_1,\ldots, B_k)\pmod{4},$$
regardless of the $3$-coloring $\psi$.  Furthermore,
$$\sum_{i=1}^k \delta_\psi(B_i)=3\sum_{i=1}^k \omega_\psi(B_i)\equiv -\sum_{i=1}^k \omega_\psi(B_i)\pmod{4}.$$
Therefore, we get the following necessary condition for the existence of a $3$-coloring.

\begin{observation}\label{obs-gen}
Let $G$ be a quadrangulation of a non-orientable surface $\Sigma$.  Let $B_1$, \ldots, $B_k$ be directed boundary cycles of $G$.  If $\psi$ is
a $3$-coloring of $G$, then
$$\Bigl(\sum_{i=1}^k \omega_\psi(B_i)\Bigr)\bmod 4=p(G,B_1,\ldots, B_k).$$
\end{observation}

If a $3$-coloring $\psi$ of the boundary cycles satisfies the condition of Observation~\ref{obs-gen}, we say that $\psi$ is
\emph{parity-compliant}.
The analogue of the result of Hutchinson~\cite{locplanq} for non-orientable surfaces without boundary was obtained by Mohar and Seymour~\cite{MohSey} and Nakamoto, Negami and Ota~\cite{NakNegOta}:
For every non-orientable surface $\Sigma$ without boundary, there exists $c$ such that a quadrangulation $G$ of $\Sigma$ without non-contractible cycles of length at most $c$ is $3$-colorable if and
only if $p(G)=0$.  Let us remark that the condition $p(G)=0$ is equivalent to stating that \emph{orienting cycles} in $G$ (cycles $K$ such that cutting $\Sigma$ along $K$ results in an orientable surface)
have even lengths.

The core of our algorithm is based on the fact that the winding-number conditions are not only necessary, but also sufficient for a precoloring of boundary cycles to extend, unless
the graph contains a small subgraph $H$ such that cutting along $H$ simplifies the surface.  Let $G$ be $2$-cell embedded in a surface $\Sigma$.  A subgraph $H$ of $G$ is \emph{non-essential}
if there exists $\Lambda\subset \Sigma$ containing $H$, where $\Lambda$ is either an open disk, or an open disk with a hole whose boundary is equal to a cuff $C$ of $\Sigma$.
A subgraph $H$ of $G$ is \emph{essential} if it is not non-essential.
A cycle $K$ is \emph{contractible} if there exists a closed disk $\Delta\subseteq \Sigma$ with boundary equal to $K$.  For a cuff $C$, let $\Sigma+\hat{C}$ denote the
surface obtained from $\Sigma$ by adding an open disk disjoint from $\Sigma$ and with boundary equal to $C$.  A cycle $K$ \emph{surrounds a cuff $C$} if
$K$ is not contractible in $\Sigma$, but it is contractible in $\Sigma+\hat{C}$.
An embedding of $G$ is \emph{boundary-linked} if
\begin{itemize}
\item $\Sigma$ is a disk and $G$ contains no path $P$ with both ends $u$ and $v$ in the boundary cycle $B$, such that both paths between $u$ and $v$ in $B$
are strictly longer than $P$; or,
\item $\Sigma$ is a cylinder with boundary cycles $B_1$ and $B_2$ and the length of every non-contractible cycle in $G$ distinct from $B_1$ and $B_2$
is \emph{strictly greater} than $\max(|B_1|,|B_2|)$; or,
\item $\Sigma$ is neither a disk nor a cylinder and for every boundary cycle $B$ of $G$, every cycle $K$ of $G$ surrounding the cuff incident with $B$ has length \emph{at least} $|B|$.
\end{itemize}
For an integer $\nu\ge 0$, we say the embedding of $G$ is \emph{$\nu$-generic} if it is boundary-linked and
every connected essential subgraph of $G$ has at least $\nu$ edges.

We say that a coloring $\psi$ of the boundary cycles of a quadrangulation of a surface $\Sigma$ \emph{satisfies the winding number constraint}
if either
\begin{itemize}
\item $\Sigma$ is orientable and the sum of winding numbers of $\psi$ on the boundary cycles of $G$ in their clockwise orientation is $0$, or
\item $\Sigma$ is non-orientable and $\psi$ is parity-compliant.
\end{itemize}

\begin{theorem}\label{thm-main}
For every surface $\Sigma$ and every integer $k$, there exists an integer $\nu$ such that the following claim holds.
For every $\nu$-generic quadrangulation $G$ of $\Sigma$ such that each boundary cycle of $G$ has length at most $k$,
a $3$-coloring $\psi$ of the boundary cycles of $G$ extends to a $3$-coloring of $G$ if and only if $\psi$ satisfies the winding number constraint.
\end{theorem}

In order to apply Theorem~\ref{thm-main}, the following characterization of connected essential subgraphs by Robertson and Seymour~\cite{rs7} is useful.
A {\em dumbbell} is a graph consisting either of two cycles with exactly one vertex in common, or two disjoint cycles and a path joining the two cycles and disjoint from them except for its ends.
A {\em theta graph} is a graph consisting of three internally disjoint paths joining the same pair of distinct vertices.
A {\em lollipop} is a graph consisting of a cycle and a path (the path may be just a single vertex) with one end on the cycle and otherwise disjoint from the cycle. The other end of the
path will be called the {\em tip} of the lollipop.
A path $P$ is a \emph{spoke} if it intersects the boundary exactly in its endpoints
and both of them belong to the same boundary cycle~$B$.  A \emph{base} of the spoke is a subpath $Q$ of $B$ with the same endpoints as $P$ such that $Q$ is homotopic to $P$.
In a disk, every spoke has two bases, while in any other surface, it has at most one base.

\begin{lemma}\label{lemma-schism}
Let $G$ be a quadrangulation of a surface $\Sigma$ and let $H$ be an inclusion-wise minimal connected essential subgraph of $G$.
Then $H$ satisfies one of the following conditions:
\begin{enumerate}[(i)]
\item $H$ is a path joining boundary vertices of distinct cuffs and containing no other boundary vertices, or
\item $H$ is a cycle containing at most one boundary vertex and $H$ is neither contractible nor surrounds a cuff, or
\item $H$ is a spoke with no base, or
\item $\Sigma$ is not the cylinder, $H$ is a dumbbell or a theta graph containing no boundary vertices, and the cycles of $H$ surround pairwise distinct cuffs of $\Sigma$, or
\item $\Sigma$ is not the cylinder, $H$ is a lollipop with the tip in a cuff $C$ and with no other boundary vertex, and the cycle of $H$ surrounds a cuff distinct from~$C$.
\end{enumerate}
If $\Sigma$ is a cylinder, then only (i) is possible.
\end{lemma}

Observe that cutting along a connected essential subgraph simplifies the surface, decreasing its genus or the number of cuffs.  In combination
with the results of~\cite{trfree4}, Theorem~\ref{thm-main} implies an important structural observation.
For a quadrangulation $G$ of a surface $\Sigma$, a subgraph $H$ of $G$ and a face $h$ of $H$, let $\Sigma_h$ denote a surface whose
interior is homeomorphic to $h$, let $\theta_h:\Sigma_h\to \Sigma$ be a continuous function whose restriction to the interior of $\Sigma_h$
is a homeomorphism to $h$, and let $G_h=\theta_h^{-1}(G)$.

\begin{theorem}\label{thm-struct}
For every surface $\Sigma$ and an integer $k\ge 0$, there exists a constant $\beta$ with the following property.  Let $G$ be a triangle-free graph embedded in $\Sigma$
so that every cuff of $\Sigma$ traces a cycle in $G$ and so that the sum of the lengths of the boundary
cycles of $G$ is at most $k$.  Suppose that every contractible $4$-cycle in $G$ bounds a face.  Then $G$ has a subgraph $H$ with at most $\beta$
vertices, such that $H$ contains all the boundary cycles and each face $h$ of $H$ satisfies one of the following.
\begin{enumerate}[(a)]
\item Every precoloring of the boundary of $h$ extends to a $3$-coloring of $G_h$, or
\item $G_h$ is a quadrangulation and every precoloring of the boundary of $H$ which satisfies the winding number constraint
extends to a $3$-coloring of $G_h$, or
\item $h$ is an open cylinder and $G_h$ is its quadrangulation, or
\item $h$ is an open cylinder and both boundary cycles of $h$ have length exactly $4$.
\end{enumerate}
\end{theorem}

Let us remark that if $\Sigma$ is the disk, the outcome (d) of the theorem cannot occur, since otherwise $G$ would contain
a non-facial contractible $4$-cycle.
As a corollary of Theorem~\ref{thm-struct}, we obtain a strengthening of the results of Hutchinson~\cite{locplanq}, Mohar and Seymour~\cite{MohSey} and Nakamoto, Negami and Ota~\cite{NakNegOta}.
Recall that the \emph{edge-width} of an embedding of a graph $G$ in a surface is the length of a shortest non-contractible cycle in $G$.

\begin{corollary}\label{cor-ew}
For every surface $\Sigma$ without boundary, there exists a constant $\gamma$ with the following property.
Let $G$ be a triangle-free graph embedded in $\Sigma$.  If $\Sigma$ is non-orientable, assume furthermore that
no subgraph $H$ of $G$ is a quadrangulation with $p(H)\neq 0$.  If the edge-width of the embedding of $G$
is at least $\gamma$, then $G$ is $3$-colorable.
\end{corollary}

In the previous paper of the series~\cite[Lemma~4.6]{trfree5}, we proved Theorem~\ref{thm-main} in the special case that $\Sigma$ is a cylinder.
In Section~\ref{sec-spec}, we prove a strengthening of Theorem~\ref{thm-main} in the special case that $\Sigma$ is a disk.
In Section~\ref{sec-main}, we prove the general case of Theorem~\ref{thm-main}, and
in Section~\ref{sec-alg}, we use it to derive the algorithm of Theorem~\ref{thm-mainalg}.
Finally, in Section~\ref{sec-struct}, we prove Theorem~\ref{thm-struct} and show that it implies Corollary~\ref{cor-ew}.

\section{The disk}\label{sec-spec}

While in general Theorem~\ref{thm-main} only gives a sufficient condition for the existence of a $3$-coloring,
in the special case of a quadrangulation of the disk, we can give an exact characterization.

\begin{lemma}\label{lem-disk}
Let $G$ be a quadrangulation of a disk $\Sigma$ with the boundary cycle $B$.  Let $\psi$ be a $3$-coloring of $B$ with winding number $0$.  Then $\psi$ extends to a $3$-coloring of $G$ if and only if
\begin{itemize}
\item[{\rm($\star$)}] every spoke $P$ of $\Sigma$ and each base $Q$ of $P$ satisfy $|P|\ge |\delta_\psi(Q)|$.
\end{itemize}
\end{lemma}
\begin{proof}
Let us first consider the case that $B$ contains a spoke $P$ and its base $Q$ such that $|P|<|\delta_\psi(Q)|$.
Let $H$ be the subgraph of $G$ drawn in the closed disk $\Delta$ bounded by the cycle $C=Q\cup P$.  Suppose that $\psi$ extends to a $3$-coloring $\varphi$ of $G$,
and consider the restriction of $\varphi$ to $H$.  We have $\omega_\varphi(C)=\frac{1}{3}\delta_\varphi(C)=\frac{1}{3}(\delta_\psi(Q)+\delta_\varphi(P))$.  However,
$|\delta_\psi(Q)+\delta_\varphi(P)|\ge |\delta_\psi(Q)|-|\delta_\varphi(P)|\ge |\delta_\psi(Q)|-|P|>0$, and thus $\varphi$ does not have zero winding number on $C$.  This contradicts
Corollary~\ref{cor-win0}, and thus no $3$-coloring extends $\psi$.

Next, let us consider the case that ($\star$) holds.  Let $B=b_1b_2\ldots b_k$, let $S$ be the set of edges $b_ib_{i+1}\in E(B)$ such that $\psi(b_{i+1})-\psi(b_i)\in \{1,-2\}$, where $b_{k+1}=b_1$,
and let $T=E(B)\setminus S$.  Note that since $\psi$ has zero winding number on $B$, we have $|S|=|T|$.
Let $\Pi$ be the sphere obtained from the disk containing $G$ by adding a disjoint open disk $\Lambda$ with the boundary $B$.
Let $G^\star$ be the dual of $G$ in $\Pi$.  Let $S^\star$ and $T^\star$ denote the sets of edges of $G^\star$ corresponding to the edges
of $S$ and $T$, respectively.  Let $H$ be the graph obtained from $G^\star$ by splitting the vertex corresponding to $\Lambda$ to two non-adjacent vertices $s$ and $t$, where $s$ is incident
with the edges of $S^\star$ and $t$ is incident with the edges of $T^\star$.

Suppose for a contradiction that $H$ contains an edge-cut $K$ separating $s$ from $t$ such that $|K|<|S^\star|$.
Let us choose such a cut $K$ for which the set $K_0$ of edges of $K$ not incident
with $\{s,t\}$ is as small as possible.
Let $Z_s$ and $Z_t$ be connected components of $H-\{s,t\}-K$ incident with an edge of $S^\star\setminus K$ and an edge of $T^\star\setminus K$, respectively.
Let $K_1$ be the set of edges of $K_0$ incident with a vertex of $Z_s$, and let $K_2$ be the set of edges of $K$ from $Z_s$ to $t$.
Since $H-\{s,t\}$ is connected, it follows that $K_1\neq\emptyset$.  Let $S_1$
be the set of edges of $S^\star$ from $s$ to $Z_s$.  Note that $K'=(K\setminus (K_1\cup K_2))\cup S_1$ is an edge-cut separating $s$ from $t$.
Since $K_1$ is nonempty and $K$ contains as few edges
not incident with $s$ or $t$ as possible among the edge-cuts of size less than $|S^\star|$, $K'$ must have size at least $|S^\star|$.
It follows that $|K'|>|K|$, and thus $|S_1|>|K_1|+|K_2|$.  Furthermore, note that
$K''=K_1\cup K_2\cup (S^\star\setminus S_1)$ is an edge-cut separating $s$ from $t$ and $|K''|=|S^\star|-|S_1|+|K_1|+|K_2|<|S^\star|$.
Since $K$ was chosen with the number $|K_0|$ of edges not incident with $\{s,t\}$ minimum and $K_1\subseteq K_0$, it follows that $K_0=K_1$; hence,
all edges of $K_0$ are incident with $Z_s$.
Symmetrically, all edges of $K_0$ are incident with $Z_t$, and thus $H-\{s,t\}-K$ has exactly two components $Z_s$ and $Z_t$.  Let $P$ be the subgraph of $G$ with edges corresponding to those in $K_0$.
Since $H-\{s,t\}-K$ has exactly two components, the graph $B+P$ (drawn in the disk $\Sigma$) has exactly two faces, i.e., $P$ is a spoke of $B$.  Let $Q$ be the subpath of $B$ whose edges correspond to $K_2\cup S_1$.
Note that $||S_1|-|K_2||=|\delta_\psi(Q)|$.  However, $|P|=|K_0|=|K\setminus ((S^\star\setminus S_1)\cup K_2)|=|K|-|S^\star|+|S_1|-|K_2|<|S_1|-|K_2|\le|\delta_\psi(Q)|$, which contradicts ($\star$).

We conclude that every edge-cut in $H$ separating $s$ from $t$ has size at least $|S^\star|$.  By Menger's theorem, $H$ contains pairwise edge-disjoint paths $P_1$, \ldots, $P_{|S^\star|}$ joining $s$ with $t$.
Note that all vertices of $H'=H-E(P_1\cup P_2\cup\ldots\cup P_{|S^\star|})$ have even degree, and thus $H'$ is a union of pairwise edge-disjoint cycles $C_1$, \ldots, $C_m$.
For $1\le i\le m$, direct the edges of $C_i$ so that all vertices of $C_i$ have outdegree $1$.  For $1\le i\le |S^\star|$, direct the edges of $P_i$ so that all its vertices except for $t$ have outdegree $1$.
This gives an orientation of $H$ such that the indegree of every vertex of $V(H)\setminus \{s,t\}$ equals its outdegree, $s$ has indegree $0$ and $t$ has outdegree $0$.
Consider the corresponding orientation of $G^\star$; since $|S^\star|=|T^\star|$, the indegree of the vertex corresponding to $\Lambda$ equals its outdegree.  Therefore, the orientation
defines a nowhere-zero $Z_3$-flow in $G^\star$.  By Tutte~\cite{tutteflow}, this nowhere-zero $Z_3$-flow corresponds to a $3$-coloring $\varphi$ of $G$, and the orientations of the edges of $S^\star$ and $T^\star$
were chosen so that $\varphi\restriction B$ is equal to $\psi$ (up to a permutation of colors).
\end{proof}

Note that if $G$ is boundary-linked, then every spoke $P$ has a base $Q$ such that $|P|\ge |Q|$,
and thus $|P|\ge |Q|\ge |\delta_\psi(Q)|$.  Moreover, consider the other base $Q'$ of $P$;
if $\psi$ has winding number zero on the boundary cycle, then $|\delta_\psi(Q')|=|\delta_\psi(Q)|$,
and thus also $|P|\ge |\delta_\psi(Q')|$.  Therefore, Lemma~\ref{lem-disk} implies Theorem~\ref{thm-main} when $\Sigma$ is a disk.

\begin{corollary}\label{cor-disk}
Let $G$ be a boundary-linked quadrangulation of a disk.
Let $\psi$ be a $3$-coloring of the boundary cycle $B$ of $G$.
Then $\psi$ extends to a $3$-coloring of $G$ if and only if it satisfies the winding number constraint (that is, the winding number of $\psi$
on $B$ is $0$).
\end{corollary}

\section{General surfaces}\label{sec-main}

We will need the following simple observation.

\begin{lemma}\label{lemma-pathcol}
Let $R=r_0r_1\ldots r_n$ be a path of even length, let $w$ be an even integer such that $|w|\le n$, and let $c_0, c_n\in\{1,2,3\}$
satisfy $c_n\equiv c_0 + w\pmod{3}$.  Then there exists a $3$-coloring $\varphi:V(R)\to \{1,2,3\}$ such that
$\varphi(r_0)=c_0$, $\varphi(r_n)=c_n$ and $\delta_\varphi(R)=w$.
\end{lemma}
\begin{proof}
We prove the claim by induction on $n$.  If $n>|w|$, then by the induction hypothesis, there exists a $3$-coloring $\varphi$
of $R'=R-\{r_0,r_1\}$ such that $\varphi(r_2)=c_0$, $\varphi(r_n)=c_n$ and $\delta_\varphi(R')=w$.
To obtain a requested $3$-coloring of $R$, we set $\varphi(r_0)=c_0$ and choose $\varphi(r_1)$ arbitrarily.

Hence, we can assume that $n=|w|$.
Then, the unique $3$-coloring $\varphi$ such that $\varphi(r_0)=c_0$ and $\delta_\varphi(r_i,r_{i+1})=\text{sgn}(w)$
for $i=0,\ldots, n-1$
satisfies $\varphi(r_n)=c_n$ and $\delta_\varphi(R)=w$.
\end{proof}

We also need the following observation on winding numbers.  It is proved by considering the cancellation
on the non-boundary edges of $H$; the walk $W$ passes twice through these edges, in the orientable case necessarily in
the opposite directions.
\begin{observation}\label{obs-windin}
Let $G$ be a quadrangulation of a surface $\Sigma$ with boundary cycles $B_1$, \ldots, $B_k$,
oriented in the clockwise direction if $\Sigma$ is orientable and arbitrarily otherwise.
Let $H$ be a subgraph of $G$ containing all boundary cycles such that $H$ has exactly one face $f$ and this
face is $2$-cell. Let $W$ be the closed walk bounding $f$, in the counterclockwise direction around $f$ in the case $\Sigma$
is orientable.  Let $\varphi$ be a $3$-coloring of $H$ and let $\psi$ be the restriction of $\varphi$ to the boundary cycles.
\begin{itemize}
\item If $\Sigma$ is orientable, then $\omega_\varphi(W)=\sum_{i=1}^k \omega_\psi(B_i)$.
\item If $\Sigma$ is non-orientable and $d$ denotes the number of boundary edges traversed by $W$ in the same direction
as the corresponding boundary cycle plus the number of non-boundary edges traversed by $W$ twice in the same direction,
then $p(G,B_1,\ldots, B_k)=2d\bmod 4$ and $\delta_\varphi(W)+2d\equiv\omega_\varphi(W)+2d\equiv \sum_{i=1}^k \omega_\psi(B_i) \pmod 4$.
\end{itemize}
\end{observation}

\noindent We say a $3$-coloring of a cycle $C$ is \emph{tamed} by a set $T$ of edges if
\begin{itemize}
\item each two vertices at distance exactly three in $C$ have different colors, and
\item $C-T$ is a union of paths of length at least one and only two colors are used on each of the paths.
\end{itemize}
For an integer $k$, we say the coloring is \emph{$k$-tamed} if additionally $|E(C)\cap T|\le k$.
A spoke of $C$ is \emph{$T$-long} if it has a base containing at least two edges of $T$.
Let us note the following property of tamed colorings.
\begin{observation}\label{obs-tamedtwo}
Let $\psi$ be a $3$-coloring of a cycle $C$ and let $B$ be a subpath of~$C$.  If $\psi$ is tamed
by a set $T$ of edges and $|E(B)\cap T|\le 1$, then $|\delta_\psi(B)|\le 2$.
\end{observation}

Let us now prove a weak variant of Theorem~\ref{thm-main}.  Let us remark that
the first two conditions from the statement of the following lemma imply that
every connected essential subgraph has at least $\rho$ edges; however, since
the first assumption applies also to the cycles that surround cuffs, these
assumptions are significantly more restrictive than $\rho$-genericity.

\begin{lemma}\label{lemma-weak}
For every surface $\Sigma$ and every integer $k$, there exist integers $d$ and $\rho$ such that the following claim holds.
Let $G$ be a quadrangulation of $\Sigma$ and $T$ a set of edges of $G$ such that
\begin{itemize}
\item[(a)] every non-contractible cycle in $G$ has length at least $\rho$,
\item[(b)] every path in $G$ of length less than $\rho$ intersecting the boundary exactly in its ends is a spoke with a base
(and in particular both its ends are incident with the same cuff), and
\item[(c)] every $T$-long spoke of a boundary cycle of $G$ has length at least $d$.
\end{itemize}
Let $\psi$ be a $3$-coloring of the boundary cycles of $\Sigma$.
If $\psi$ is $k$-tamed by $T$ on every boundary cycle and satisfies the winding number constraint, then it extends to a $3$-coloring of~$G$.
\end{lemma}
\begin{proof}
Let $g$ be the Euler genus of $\Sigma$ and $c$ the number of cuffs of $\Sigma$.
By Gr\"otzsch's Theorem, Theorem~\ref{thm:gimtho} and Corollary~\ref{cor-disk}, the claim holds (with $d=\rho=0$)
if $\Sigma$ is the sphere, the projective plane or the disk.  Hence, assume that $g+c\ge 2$.
Let $b=g+c-1$, $\mu=12b+2kc+2$, $\beta=12b\mu$, $\lambda=\mu+2\beta$, $\rho=2b\lambda$ and $d=\beta+4\mu$.

Let $H$ be a subgraph of $G$ with as few edges as possible such that $H$ contains all boundary cycles of $G$ and every face of $H$ is homeomorphic to an open disk.
Observe that $H$ is connected and has exactly one face $\Lambda$.  By cutting along $H$, we obtain an embedding of a graph $G'$ in a closed disk $\Delta$ together with a continuous surjection $\theta:\Delta\to\Sigma$
mapping $G'$ to $G$ such that the restriction of $\theta$ to the interior of $\Delta$ is a homeomorphism to $\Lambda$.  Furthermore, $\theta$ maps the cycle $\Gamma$ of $G'$ forming the boundary of $\Delta$
to the boundary walk $W$ of the face $\Lambda$ of $H$.

Since $g+c\ge 2$, $H$ has minimum degree at least two and $H$ is not a cycle.  Let $X$ be the set of vertices of $H$ of degree at least three
and let $\PP$ be the set of all subgraphs $P$ of $H$ such that
either $P$ is a path in $H$ joining two vertices $u,v\in X$, or $P$ is a cycle containing a vertex $u=v\in X$,
and such that no other vertex of $P$ belongs to $X$ and $P$ is not a part of a boundary cycle.
If $P$ has length less than $\lambda$, then let $M(P)$ be the null graph.  Otherwise, let $M(P)$ be a subpath of $P$ of length $\mu$
chosen so that the distance between the endvertices of $M(P)$ and $\{u,v\}$ in $P$ is at least $\beta$. 
Let $M$ be the union of $M(P)$ over all $P\in\PP$.

Note that by Euler's formula, $H$ has at most $2b$ vertices of degree at least three and that $|\PP|\le 3b$.
Let $L$ be obtained from $H$ by removing edges of $M$ and of the boundary cycles, and then removing isolated vertices.
Consider a path $Q$ in $H$ intersecting the boundary of $\Sigma$ exactly in its ends.
Note that if both ends of $Q$ belong to the same boundary cycle $C$, then $C+Q$ does not contain a contractible cycle,
since $H$ has only one face, and thus the spoke $Q$ has no base.  By (b), the path $Q$ has length at least
$\rho>(|X|-1)(\lambda-1)$, and thus $Q$ contains a subpath of length at least $\lambda$ with no internal vertices belonging to $X$.
We conclude that $M$ intersects $Q$, and thus $L$ contains no path intersecting the boundary of $\Sigma$
exactly in its ends.  Similarly, since $H$ has only one face, every cycle in $H$ is non-contractible, and thus has length at least $\rho$
by (a).  It follows that every cycle in $H$ has an edge belonging to either the boundary of $\Sigma$ or $M$.
Therefore, $L$ is a union of trees, each of them intersecting the boundary of $\Sigma$ in at most one vertex.

If $\Sigma$ is orientable, then let $M'=M$.  If $\Sigma$ is non-orientable, then there exists a path $R$ forming a component of $M$ such that
the walk $W$ traverses $R$ twice in the same direction. In this case, we set $M'=M-R$.

We now define a $3$-coloring $\varphi$ of $H$; this will also give a $3$-coloring $\varphi'$~of~$\Gamma$ such that
$\varphi'(x)=\varphi(\theta(x))$ for each $x\in V(\Gamma)$.  For each vertex $v$ incident with the boundary of $\Sigma$, we let
$\varphi(v)=\psi(v)$.  Let $\QQ$ be the set of paths of $G-T$ drawn in the boundary of $\Sigma$; since $\psi$
is tamed by $T$ on each boundary cycle, $\psi$ uses exactly two colors on each path in $\QQ$.
We extend $\varphi$ to $L$ so that each component of $L$ and the path of $\QQ$ that it intersects (if any)
is colored by exactly two colors.  Next, for each path $v_0v_1\ldots v_{\mu}$ of $M'$, we extend the coloring
of the component of $L$ containing $v_0$ to $v_0v_1\ldots v_{\mu-2}$ using the same two colors, and we choose $\varphi(v_{\mu-1})$
distinct from $\varphi(v_{\mu-2})$ and $\varphi(v_\mu)$.

Finally, in the case that $\Sigma$ is non-orientable, we need to determine the coloring of $R$.
Let $P_1$ and $P_2$ be the two paths obtained from $\Gamma$ by removing the edges and the internal vertices of $\theta^{-1}(R)$.
For $i\in \{1,2\}$, let $u_i$ and $v_i$ be the first and the last vertex of $P_i$, respectively,
let $w_i=\delta_{\varphi'}(P_i)$ and let $w=w_1+w_2$.  Since $\psi$ is parity-compliant, $|R|=\mu$ is even, and the walk $W$
traverses $R$ twice in the same direction, Observation~\ref{obs-windin} implies that $w$ is divisible by $4$.
Moreover, since $W=\theta(\Gamma)$ traverses $R$ twice in the same direction, we have $\theta(u_1)=\theta(u_2)$ and $\theta(v_1)=\theta(v_2)$,
and thus $\varphi'(u_1)=\varphi'(u_2)$ and $\varphi'(v_1)=\varphi'(v_2)$.  It follows that $w_1\equiv w_2\pmod{3}$, and thus
$w_1\equiv w/2\pmod{3}$.  Furthermore, by the construction of $\varphi$, there exists a set $Y\subseteq E(P_1\cup P_2)$ of size
at most $12b+kc$ such that $\varphi'$ uses at most two colors on each component of $(P_1\cup P_2)-Y$.  Let $a$ be the number of
components of $(P_1\cup P_2)-Y$ and note that $a\le 2+|Y|$.  It follows that $|w|\le |Y|+a\le 2|Y|+2<2\mu=2|R|$.
Since $R$ has even length and $w$ is divisible by $4$, Lemma~\ref{lemma-pathcol} implies that we can extend $\varphi$ to $R$ so
that $\delta_{\varphi}(R)=-w/2$.

Note that $\omega_{\varphi'}(\Gamma)=0$, by Observation~\ref{obs-windin} if $\Sigma$ is orientable and 
by the choice of the coloring of $R$ if $\Sigma$ is non-orientable.
We claim that $\varphi'$ extends to a $3$-coloring of $G'$.  This will finish the argument, as the $\theta$-image of this $3$-coloring
gives a $3$-coloring of $G$ extending $\psi$.

Suppose for a contradiction that $\varphi'$ does not extend to a $3$-coloring of $G'$.
By Lemma~\ref{lem-disk}, there exists a spoke $P$ of $\Gamma$ which has a base $Q$ such that $|P|<|\delta_{\varphi'}(Q)|$.
By the construction of $\varphi'$, we have $|\delta_{\varphi'}(Q)|\le 4\mu$, and thus $|P|<4\mu$.
Let $s$ and $t$ be the endpoints of $P$.
Let $F$ be the subgraph of $H$ formed by the edges $e\in E(\theta(Q))$ such that either $e$ is a part of a boundary cycle, or
exactly one of the two edges of $\theta^{-1}(e)$ belongs to $Q$.
Note that $F$ is obtained from the walk $\theta(Q)$ by removing edges that appear twice in the walk.
Moreover, the walk $\theta(Q)$ enters and leaves each vertex other than $s$ and $t$ the same number of times.
Hence, if $\theta(s)\neq\theta(t)$, then $\theta(s)$ and $\theta(t)$ have odd degree in $F$
and all other vertices of $F$ have even degree.  Consequently, $F$ contains a path $Q'$ joining $\theta(s)$ with $\theta(t)$.

For any path $S\in \PP$, we have $|E(S\cap Q')|\le|P|$, since otherwise $H-(S\cap Q')+\theta(P)$ has fewer edges than $H$ and it has
only one face homeomorphic to a disk (since only one of the paths $\theta^{-1}(S\cap Q')$ belongs to $Q$), contrary to the choice of $H$.
Thus, at most $3b|P|<12b\mu\le \beta$ edges of $Q'$ do not belong to the boundary of $\Sigma$.  We conclude that
if $Q'\cap M\neq \emptyset$, then $Q'$ is a subpath of a path in $\PP$.

If $Q'\cap M\neq \emptyset$, then let $Q_1=Q'$.
Otherwise, either $Q'\subseteq L$, or there exists a boundary cycle $C$
such that $Q'\subset C+L$.  If no edge of $Q'$ belongs to a boundary cycle, then let $Q_1=Q'$.
If some edge of $Q'$ belongs to the boundary cycle $C$,
then $Q'\cup \theta(P)$ contains a spoke $A$ of $C$, and since every spoke without a base has at least $\rho$ edges by (b),
this spoke has a base $B$.  In this case, we let $Q_1=(A\cup B)\cap (C\cup L)$.
Note that $|A|=|E(Q')\setminus E(C)|+|P|<\beta+4\mu=d$.

Let $K=Q_1+\theta(P)$.  Note that $K$ is a contractible cycle; when no edge of $Q'$ belongs to a boundary cycle, this follows by (a),
since $|K|<\rho$.
Let $\Lambda'\subset \Sigma$ be the open disk bounded by $K$.  Note that $\theta^{-1}(\Lambda')$ is one of the two faces of $\Gamma+P$,
and thus $|\delta_\varphi(Q_1)|=|\delta_{\varphi'}(Q)|>|P|$.
This implies that $|Q_1|>|P|$.  Therefore, $Q_1$ is not a subpath of a path in $\PP$,
as otherwise we could replace $Q_1$ by $P$ in $H$, contradicting the minimality of $H$.
It follows that $Q'\cap M=\emptyset$.

Since $G$ is a quadrangulation, $|Q_1|$ and $|P|$ have the same parity, and because
$\delta_\varphi(Q_1)$ and $|Q_1|$ have the same parity, we have $|\delta_\varphi(Q_1)|\ge |P|+2$.
Therefore, $|\delta_\varphi(Q_1)|>2$.  Since $\varphi$ uses only two colors on each component of $L$,
it follows thay $Q_1$ is not a subpath of a single component of $L$, and thus $Q_1$ must intersect a boundary cycle.
Hence, $K$ is a concatenation of the spoke $A$ with the base $B$.
Since $|\delta_\varphi(Q_1)|>2$, the construction of $\varphi$ and Observation~\ref{obs-tamedtwo} imply
$|E(B)\cap T|\ge 2$, and thus the spoke $A$ is $T$-long.  However, $|A|<d$, contradicting the assumptions.
\end{proof}

We now aim to eliminate the extra assumptions from the statement of Lemma~\ref{lemma-weak}, most importantly the one that
non-contractible cycles that surround cuffs have length at least $\rho$.  To this end, we employ ideas developed in~\cite{rs7}.
In particular, we need the following lemma (compared to the statement in~\cite{rs7}, we take advantage of the fact that we
deal with quadrangulations, and thus we can state the assumptions in terms of the paths and cycles in the graph rather than in terms
of curves in the surface intersecting the graph only in vertices).

\begin{lemma}[{Robertson and Seymour~\cite[(5.8)]{rs7}}]\label{lemma-grid}
Let $H$ be a quadrangulation of the cylinder with boundary cycles $B_1$ and $B_2$, and let $r,s\ge 1$ be integers.  Suppose that the distance between $B_1$ and $B_2$
in $H$ is at least $2r-3$ and that every non-contractible cycle in $H$ has length at least $2s$.  Then $H$ contains
pairwise vertex-disjoint non-contractible cycles $C_1$, \ldots, $C_r$ and pairwise vertex-disjoint paths $P_1$, \ldots, $P_s$ from
$B_1$ to $B_2$ such that $C_i\cap P_j$ is a path for all $i\in\{1,\ldots,r\}$ and $j\in\{1,\ldots,s\}$.
\end{lemma}

For a quadrangulation $G$ of a surface, let $\zeta(G)$ denote the minimum number of edges of a connected essential subgraph of $G$.
The following lemma gives us a cylinder around a cuff to which we can later apply Lemma~\ref{lemma-grid}.

\begin{lemma}\label{lemma-surcyc}
Let $a\ge 1$ be an integer, let $G$ be a quadrangulation of a surface $\Sigma$, where $\Sigma$ has non-zero genus or at least two cuffs,
let $B$ be a boundary cycle of $G$, and let $C$ be the incident cuff.
If $\zeta(G)\ge 2a+4$, then there exists a cycle $K$ in $G$ surrounding $C$ such that for each vertex $v\in V(K)$,
the distance from $v$ to $B$ is $a$ or $a+1$.  Moreover, for every non-contractible cycle $K'$ in $G$ drawn in the part of $\Sigma$
between $B$ and $K$ (inclusive), denoting by $\Sigma'$ the surface obtained by deleting the part of $\Sigma$ between $B$ and $K'$
(including $B$, but excluding $K'$), and letting $G'$ be the subgraph of $G$ drawn in $\Sigma'$, we have $\zeta(G')\ge \zeta(G)-4a-6$.
\end{lemma}
\begin{proof}
Let $M'$ be the set of vertices of $G$ at distance at most $a+1$ from $B$, and let $M\subset M'$ consist of the vertices
at distance $a$ or $a+1$ from $B$.  For each $v\in M'$, choose a shortest path $P_v$ from $B$
to $v$ in $G$ arbitrarily.  For any edge $uv\in E(G[M'])$, let $B_{uv}$ be a path in $B$ from the end of $P_u$ to the end of $P_v$,
let $W'_{uv}$ be the walk consisting of $P_u$, the edge $uv$ and the reversal of $P_v$, and let $W_{uv}$ be the closed walk
consisting of $W'_{uv}$ and the reversal of $B_{uv}$.
Since $|E(W'_{uv})|\le 2a+3<\zeta(G)$, $W'_{uv}$ does not contain a connected essential subgraph, and thus $W_{uv}$ is contractible
in $\Sigma+\hat{C}$.  Consequently, any cycle $Q$ in $G[M']$ is homotopic in $\Sigma+\hat{C}$ to the concatenation of the paths $B_{uv}$
for $uv\in E(Q)$, and thus $Q$ is either contractible or surrounds $C$ (and in particular, $G[M']$ does not contain
any connected essential subgraph).  Conseqently, this is also the case for the cycles in $G[M]$.

Suppose for a contradiction no cycle in $G[M]$ surrounds $C$, and thus all cycles in $G[M]$ are contractible.  Then there exists
a disjoint union $\Lambda$ of finitely many open disks in $\Sigma$ such that $G[M]\subset \Lambda$.  Observe there exists a simple curve $c$
in $\Sigma$ disjoint from $\Lambda$ and $V(G)$ such that $c$ intersects the boundary of $\Sigma$ exactly in its ends,
one end of $c$ is in $C$ and either the other end of $c$ is in another cuff of $\Sigma$
(if $\Sigma$ has at least two cuffs), or the other end of $c$ is in $C$ and no simple closed curve in $C\cup c$ is contractible (if $\Sigma$ has
non-zero genus).  Let $X$ be the set of edges of $G$ crossed by $c$.  Each edge of $X$ is incident with a vertex at distance strictly
less than $a$ from $B$: Indeed, suppose this is not the case and consider the first edge $e\in X$ along $c$ such that both ends of $e$
are at distance at least $a$ from $B$.  Since $e\not\subset \Lambda$, we have $e\not\in E(G[M])$, and thus at least one end of $e$ is at distance at least $a+2$ from $B$.
However, consider the edge $e'\in X$ that $c$ intersects before $e$.  By the choice of $e$, at least one end of $e'$ is at distance
at most $a-1$ from $B$.  However, $e$ and $e'$ are incident with the same face of length at most four, which is a contradiction.
Let $Y$ be the set of vertices incident with the edges of $X$.  As we just argued, each vertex of $Y$ is at distance at most $a$ from $B$, and thus
$G[Y]\subseteq G[M']$.  Moreover, $G[Y]$ contains a walk homotopic to $c$, and thus $G[Y]$ contains a connected essential subgraph.
This is a contradiction.

Therefore, we can choose $K$ as a cycle in $G[M]$ surrounding $C$.  Note that we can without loss of generality assume
the paths $P_v$ for $v\in V(K)$ were chosen so that they do not cross each other, i.e., so that $P_u\cap P_v$ is either
empty or a path starting in $B$ for each $u,v\in V(K)$.  Moreover, for each $uv\in E(K)$, we can without loss of generality
assume $B_{uv}$ is chosen so that the closed walk $W_{uv}$ is contractible.  Let $\Delta_{uv}$ be the union of $W_{uv}$ and
the closed disk in $\Sigma$ bounded by the cycle in $W_{uv}$.  Note that $\Delta=\bigcup_{uv\in E(K)} \Delta_{uv}$ is exactly the part of $\Sigma$
between $B$ and $K$.

Consider now a non-contractible cycle $K'$ drawn in $\Delta$, and let $\Sigma'$ and $G'$ be as in the statement of the lemma.
Let $Z'$ be a connected essential subgraph of $G'$.  Note that $Z'$ intersects $K'$ in at most two vertices.  For each vertex $x\in V(Z')\cap V(K')$,
let $e_x$ be an edge of $K$ such that $x\in \Delta_{e_x}$.  Since $Z'$ is essential, we have $Z'\not\subseteq \Delta_{e_x}$,
and thus $Z'$ intersects $W_{e_x}$.  Hence, $Z=Z'\cup \bigcup_{x\in V(Z)\cap V(K')} W_{e_x}$ is a connected essential subgraph of $G$
and $\zeta(G)\le |E(Z)|\le |E(Z')|+4a+6$.  It follows that $\zeta(G')\ge \zeta(G)-4a-6$.
\end{proof}

Next, we use these results to obtain new cuffs with tamed colorings.

\begin{lemma}\label{lemma-tame}
Let $k,d\ge 2$ and $\rho\ge 2kd+2$ be integers, let $\nu=\rho+4(k+d)-2$, and let $G$ be a quadrangulation of a surface $\Sigma$,
where $\Sigma$ has non-zero genus or at least three cuffs.  Let $B$ be a boundary cycle of $G$ of length $k$
and let $\psi$ be a $3$-coloring of $B$.  If $\zeta(G)\ge \nu$ and every cycle surrounding the cuff of $B$ has length
at least $k$, then there exists a cycle $Q$ surrounding the cuff of $B$ and a set $T\subseteq E(Q)$, such that,
denoting by $\Pi$ the part of $\Sigma$ between $B$ and $Q$ (inclusive) and letting $\Sigma'=\Sigma\setminus (\Pi\setminus Q)$,
\begin{itemize}
\item $\psi$ extends to a $3$-coloring $\varphi$ of the subgraph $G_B$ of $G$ drawn in $\Pi$ such that $\varphi$ is $k$-tamed by $T$ on $Q$,
\item no cycle in $G$ of length at most $\rho$ surrounding the cuff of $B$ is drawn in~$\Sigma'$,
\item every $T$-long spoke of $Q$ drawn in $\Sigma'$ has length at least $d$, and
\item the subgraph $G'$ of $G$ drawn in $\Sigma'$ satisfies $\zeta(G')\ge \zeta(G)-\rho-8(k+d)+6$.
\end{itemize}
\end{lemma}
\begin{proof}
Let $a=2(k+d)-3$.  Let $B_0$ be a cycle of length at most $\rho$ surrounding the cuff of $B$ such that the part $\Pi_0$ between $B$ and $B_0$ (inclusive)
is maximal, let $\Sigma_0=\Sigma\setminus (\Pi_0\setminus B_0)$, and let $G_0$ be the subgraph of $G$ drawn in $\Sigma_0$.
Since $\Sigma$ either has non-zero genus or at least three cuffs, we have $\zeta(G_0)\ge \zeta(G)-|B_0|\ge \zeta(G)-\rho\ge 2a+4$.
Let $K$ be the cycle obtained by applying Lemma~\ref{lemma-surcyc} in $G_0$, and note that the distance between $B_0$ and $K$ is
at least $a$.

Let us remark that $|B_0|\ge \rho-2\ge 2kd$, since otherwise we could replace an edge of $B_0$ by the rest of the boundary of an incident 4-face
drawn in $\Sigma_0$, contradicting the maximality of $\Pi_0$.  Hence, the choice of $B_0$ implies that every non-contractible cycle
drawn in the cylinder $\Pi_1$ between $B_0$ and $K$ has length at least $2kd$.  By Lemma~\ref{lemma-grid},
there are pairwise vertex-disjoint cycles non-contractible cycles $C_1$, \ldots, $C_{k+d}$ and
pairwise vertex-disjoint paths $P_1$, \ldots, $P_{kd}$ in $G$ drawn in $\Pi_1$, such that $C_i\cap P_j$ is a path for
all $i\in\{1,\ldots, k+d\}$ and $j\in \{1,\ldots, kd\}$.  We can also assume that for $1\le i<j\le k+d$,
the cycle $C_i$ separates $B_0$ from $C_j$, and that for $j\in \{1,\ldots, kd\}$, the path $P_j$ intersects $C_1$ and $C_{k+d}$ exactly
in its ends.  For $1\le i\le kd$, let $\Delta_i\subset \Pi_2$ be the closed disk bounded by the paths $P_i$ and $P_{i+1}$
(where $P_{kd+1}=P_1$) and by subpaths of $C_1$ and $C_{k+d}$ chosen so that $\Delta_i$ does not contain the rest of the paths
$P_1$, \ldots, $P_{kd}$.

Let us set $Q=C_{k+1}$.  
Let $\hat{\Sigma}$ be the surface obtained from $\Sigma$ by patching all holes, let $G^\star$ be the dual graph of $G$ in its
drawing in $\hat{\Sigma}$, and let $b$ be the vertex of $G^\star$ dual to the face formed by the patch over the cuff of $B$.
Let $Q^\star$ be the cut in $G^\star$ consisting of the edges dual to those in $Q$.
For $1\le i\le kd$, note that there exists a path $L^\star_i$ in $G^\star$ such that the first edge of $L^\star_i$ is dual to an edge of $C_1$,
the rest of $L^\star_i$ is drawn within $\Delta_i$, and $E(L^\star_i)\cap Q^\star$ contains exactly the last edge of $L^\star_i$.
Similarly, for $1\le i\le k$, there exists a non-contractible cycle $M^\star_i$ in $G^\star$ drawn between $C_i$ and $C_{i+1}$.
Let $T^\star$ be the set of the ends of the paths $L^\star_1$, $L^\star_{d+1}$, \ldots, $L^\star_{(k-1)d+1}$ in $Q^\star$,
and let $T\subset E(Q)$ be the set of edges of $G$ dual to those in $T^\star$.

We claim that $G^{\star}-(Q^\star\setminus T^\star)$ contains $k$ pairwise edge-disjoint paths $F^\star_1$, \ldots, $F^\star_k$ from $b$ to $T^\star$.
By Menger's theorem, it suffices to show that for every set $Y\subseteq E(G^{\star})$ of size less than $k$, the graph $G^{\star}-(Y\cup (Q^\star\setminus T^\star))$
contains a path from $b$ to $T^\star$.  Indeed, since every non-contractible cycle of $G$ has length at least~$k$, Menger's theorem
implies that $G^{\star}$ contains at least $k$ pairwise edge-disjoint paths from $b$ to $Q^\star$, and at least one of them is disjoint from $Y$;
let us denote this path by $A^\star$.  Furthermore, for some $1\le i,j\le k$, the cycle $M^\star_i$ and the path $L^\star_{d(j-1)+1}$ are disjoint from $Y$.
Then, $(A^\star-Q^\star)+M^\star_i+L^\star_{d(j-1)+1}$ is a connected subgraph of $G^{\star}-(Y\cup (Q^\star\setminus T^\star))$ containing both $b$ and an edge of~$T^\star$.

Note that we can choose the paths $F^\star_1$, \ldots, $F^\star_k$ so that they do not cross each other.
For $i\in\{1,\ldots,k\}$, let $F_i$ denote the set of edges of $G$ dual to the edges of $F^\star_i$.
Let $B=v_1v_2\ldots v_k$, where the edge $v_iv_{i+1}$ belongs to $F_i$ for $1\le i\le k$ (and $v_{k+1}=v_1$).  
Recall that $\Pi$ and $G_B$ are defined in the statement of the lemma.
For each $v\in V(G_B)$, there exists unique $i(v)\in \{1,\ldots, k\}$
such that the face of $G^\star$ dual to $v$ is drawn in the region bounded by $F^\star_{i(v)-1}$ and $F^\star_{i(v)}$,
where $F^\star_0=F^\star_k$.  Let $V_i=\{v\in V(G_B):i(v)=i\}$.  Note that the graph $G_B[V_i]$ is bipartite.
We now color $G_B[V_i]$ by colors $\psi(v_{i-1})$ and $\psi(v_i)$, where $v_0=v_k$, so that $v_i$ keeps the color $\psi(v_i)$.
Let us define the coloring more precisely and argue that it is proper. 

If $k$ is even, then since $G$ is a quadrangulation, $G_B$ is bipartite; let $\iota:V(G_B)\to \{1,2\}$ be a $2$-coloring of $G_B$.
If $k$ is odd, then there exists a $2$-coloring $\iota$ of $G_B-F_k$ such that
$\iota(u)=\iota(v)$ for each $uv\in F_k$.  In both cases, choose $\iota$ so that $\iota(v_1)=1$.
Let $f:\{1,\ldots, k\}\times\{1,2\}\to\{1,2,3\}$ be defined by $f(i,\iota)=\psi(v_i)$ if $i$ and $\iota$ have the same parity
and by $f(i,\iota)=\psi(v_{i-1})$ otherwise.  We define $\varphi:V(G_B)\to\{1,2,3\}$ by $\varphi(v)=f(i(v),\iota(v))$.
Clearly $\varphi$ extends $\psi$.

Consider an edge $uv\in E(G_B)$.  If $i(u)=i(v)$, then $\varphi(u)\neq\varphi(v)$, since $\iota(u)\neq\iota(v)$
and $f(i,1)\neq f(i,2)$ for every $i$.  Thus, we can assume that either $i(u)=i(v)+1$, or $i(u)=k$ and $i(v)=1$.  
Observe that in both cases, $i(u)-i(v)$ and $\iota(u)-\iota(v)$ have the same parity,
and thus either $\varphi(u)=\psi(v_{i(u)})$ and $\varphi(v)=\psi(v_{i(v)})$, or $\varphi(u)=\psi(v_{i(u)-1})$ and $\varphi(v)=\psi(v_{i(v)-1})$.
Therefore, $\varphi(u)\neq \varphi(v)$, and thus $\varphi$ is a $3$-coloring of $G_B$.

Note that $Q\setminus T$ consists of $k$ paths of length at least $d-1\ge 1$, and $\varphi$ uses only two colors
on each of the paths.  Furthermore, whenever $P$ is the union of two such consecutive paths and the edge of $T$ between
them, there exists a color class of $\varphi$ whose complement is an independent set in $P$.
This implies that $\varphi$ is $k$-tamed by $T$ on $Q$.

The choice of $B_0$ implies that no cycle in $G$ of length at most $\rho$ surrounding the cuff of $B$ is drawn in $\Sigma'$.
Consider any $T$-long spoke $S$ of $Q$ drawn in $\Sigma'$, and let $R\subseteq Q$ be the base of $S$ containing
at least two edges of $T$.  Then at least $d$ of the paths $P_1$, \ldots, $P_{kd}$ intersect $R$.
If $S$ does not intersect all these paths, then it cannot be contained in the cylinder between $Q$ and $C_{k+d}$,
and thus it must intersect the cycles $C_{k+2}$, \ldots, $C_{k+d}$.  Consequently, $S$ has length at least $d$.
Finally, according to Lemma~\ref{lemma-surcyc} for $K'=Q$, we have $\zeta(G')\ge \zeta(G_0)-4a-6\ge \zeta(G)-\rho-8(k+d)+6$.
\end{proof}

Finally, we need to combine Lemmas~\ref{lemma-weak} and \ref{lemma-tame}.

\begin{proof}[Proof of Theorem~\ref{thm-main}]
The conclusion of Theorem~\ref{thm-main} holds when $\Sigma$ is the sphere by Gr\"otzsch's theorem,
when $\Sigma$ is the disk by Corollary~\ref{cor-disk},
and when $\Sigma$ is the cylinder by Lemma~4.6 of \cite{trfree5}.
Therefore, we can assume that $\Sigma$ has either positive genus
or at least three cuffs.

Let $d$ and $\rho$ be the constants of Lemma~\ref{lemma-weak} applied for $\Sigma$ and $k$;
witout loss of generality, we can assume $d\ge 2$ and $\rho\ge 2kd+2$.
Let $m$ be the number of cuffs of $\Sigma$ and let $\nu=\rho+m(\rho+8(k+d)-6)$.
We already argued in the introduction (Corollary~\ref{cor-win0} and Observation~\ref{obs-gen})
that the winding number constraint is necessary for the existence of an extension of $\psi$.  Let us now prove that it is sufficient.

Let $B_1$, \ldots, $B_m$ be the boundary cycles of $G$.  For $1\le i\le m$, we apply Lemma~\ref{lemma-tame} at $B_i$,
with $|B_i|$ playing the role of $k$.  Let $Q_i$ be the corresponding cycle surrounding $B_i$, with a set $T_i\subseteq E(Q_i)$
and a $3$-coloring $\varphi_i$ of the subgraph $G_{B_i}$ of $G$ between $B_i$ and $Q_i$ such that
$\varphi_i$ is $k$-tamed by $T_i$ on $Q_i$.  Let $G'$ be the subgraph of $G$ drawn
in the subsurface $\Sigma'$ of $\Sigma$ obtained by deleting the parts between $B_i$ and $Q_i$ (including $B_i$ but excluding $Q_i$).
According to Lemma~\ref{lemma-tame}, we have $\zeta(G')\ge \zeta(G)-m(\rho+8(k+d)-6)\ge \rho$.  Consequently,
every path in $G'$ of length less than $\rho$ intersecting the boundary of $\Sigma'$ exactly in its ends is non-essential,
and thus it is a spoke with a base.  Moreover, every cycle of length less than $\rho$ in $G'$ is non-essential,
and does not surround any of the cuffs by the choice of $Q_1$, \ldots, $Q_m$, and thus is contractible.
Finally, for $T=\bigcup_{i=1}^m T_i$, the choice of $Q_1$, \ldots, $Q_m$ using Lemma~\ref{lemma-tame} ensures
that every $T$-long spoke of a boundary cycle of $G'$ has length at least $d$.

Let $\psi'$ be the restriction of $\varphi_1\cup\ldots\cup \varphi_m$ to the boundary of $\Sigma'$.
By Corollary~\ref{cor-win0} applied to $G_{B_i}$, we have $\omega_{\psi'}(Q_i)=\omega_{\psi}(B_i)$ for $i\in\{1,\ldots,m\}$.
Moreover, $p(G,B_1,\ldots,B_m)=p(G',Q_1,\ldots,Q_m)$, since for $i\in\{1,\ldots,m\}$, we can choose the orientation of the faces between $B_i$ and $Q_i$
consistently with the orientation of $B_i$, so that these faces do not contribute to the graph $D$
from the definition of $p(G,B_1,\ldots, B_m)$.  Consequently, the winding number constraint for $\psi$ in $G$
implies the winding number constraint for $\psi'$~in~$G'$.  By Lemma~\ref{lemma-weak}, $\psi'$ extends to a $3$-coloring $\varphi'$ of $G'$.
Therefore, $\psi$ extends to a $3$-coloring $\varphi'\cup\bigcup_{i=1}^m\varphi_i$ of~$G$.
\end{proof}

\section{The algorithm}\label{sec-alg}

Note that the proof of Lemma~\ref{lem-disk} gives an algorithm to decide whether the precoloring of the boundary of a disk
extends to a $3$-coloring of a quadrangulation of the disk (and to find such a coloring if it exists), by reducing the problem to
finding the maximum number of edge-disjoint paths between prescribed vertices $s$ and $t$ (and by turning the resulting $Z_3$-flow
to a $3$-coloring).  This corresponds to a network flow problem, which can be solved in linear time using Ford-Fulkerson
algorithm when the degrees of $s$ and $t$ are bounded by a constant.  Hence, we have the following.

\begin{observation}\label{obs-mainalg-disk}
Theorem~\ref{thm-mainalg} is true if $\Sigma$ is the sphere or the disk.
\end{observation}

For quadrangulations of the cylinder, we proved a result similar to Corollary~\ref{cor-disk} in the previous paper of the series.

\begin{lemma}[{Dvo\v{r}\'ak et al.~\cite[Corollary~4.7]{trfree5}}]\label{lemma-mainalg-cyl1}
For all positive integers $d_1$ and $d_2$, there exists a linear-time algorithm as follows.
Let $G$ be a boundary-linked quadrangulation of the cylinder with boundary cycles $B_1$ and $B_2$
such that $|B_1|=d_1$, $|B_2|=d_2$ and the distance between $B_1$ and $B_2$ is at least $d_1+d_2$.
Let $\psi$ be a $3$-coloring of $B_1\cup B_2$ satisfying the winding number constraint.
Then the algorithm returns a $3$-coloring of $G$ that extends $\psi$.
\end{lemma}

Next, we need an algorithm to split a graph embedded in a cylinder to smaller subgraphs.
For a graph $G$ and a set $X$ of its edges, let $G/X$ denote the graph obtained by contracting all edges in $X$.

\begin{lemma}\label{lemma-splitcyl}
Let $d$ be a positive integer.  There exists a linear-time algorithm that,
given a graph $G$ that is $2$-cell embedded in the cylinder $\Sigma$ with boundary cycles $B_1$ and $B_2$
of length at most $d$, returns a sequence $C_0$, $C_1$, \ldots, $C_m$ of non-contractible cycles of $G$ of length at most $d$
such that
\begin{itemize}
\item $C_0=B_1$ and $C_m=B_2$,
\item for $0\le i< m$, the cycle $C_i$ is contained in the part of $\Sigma$ between $B_1$ and $C_{i+1}$, and
\item either $C_i$ intersects $C_{i+1}$, or the subcylinder of $\Sigma$ between $C_i$ and $C_{i+1}$ contains no
non-contractible cycle of length at most $d$ distinct from $C_i$ and $C_{i+1}$.
\end{itemize}
\end{lemma}
\begin{proof}
For $2\le k\le d+1$, we are going to construct an algorithm $A_k$ with the same specification as in the statement
of the lemma, under the additional assumption that every non-contractible cycle of length less than $k$ shares
an edge with one of the boundary cycles.  Note that for $k=2$, the assumption is void, and thus this will give
a proof of Lemma~\ref{lemma-splitcyl}.

We proceed by induction on decreasing $k$.  First,  let us assume that $k\le d$ and that the algorithm $A_{k+1}$
exists.  Let $\hat{\Sigma}$ be the sphere obtained from $\Sigma$ by patching the holes, and let $s$ and $t$ be the faces of $G$
in its embedding in $\hat{\Sigma}$ bounded by $B_1$ and $B_2$.
Let $G^\star$ be the dual of $G$ in its embedding in $\hat{\Sigma}$, and let $s^\star$ and $t^\star$ be the vertices of $G^\star$ dual to $s$ and $t$.
Let $F_0=E(B_1)\cup E(B_2)$ and let $F_0^\star$ be the set of the edges dual to those in $F_0$.
Consider a maximum flow from $s^\star$ to $t^\star$ in $G^\star/F_0^\star$ (where all edges have capacity $1$).

The size of the flow is equal to the length of the shortest non-contractible cycle in $G-F_0$.
By the assumption that every non-contractible cycle of length less than $k$ shares
an edge with one of the boundary cycles, the flow has size at least $k$.
If the size of the flow is at least $k+1$, then the claim follows by applying the algorithm $A_{k+1}$.
Therefore, assume that the maximum flow has size exactly $k$.  Then, there exists a unique non-contractible $k$-cycle $Q_1$ in $G-F_0$
which is nearest to $s$, corresponding to the cut of size $k$ in $G^\star/F_0^\star$ bounding the set of vertices that can be reached from
$s^\star$ by augmenting paths.  Let $F_1$ consist of $F_0$ and all edges of $G-F_0$ drawn in the closed disk in $\hat{\Sigma}$ bounded by $Q_1$ that contains $s$.
Similarly, we find the non-contractible $k$-cycle $Q_2$ in $G-F_1$ which is nearest to $s^\star$, and so on, until no such $k$-cycle exists.
Hence, we obtain a sequence of pairwise edge-disjoint non-contractible cycles $Q_0$, $Q_1$, \ldots, $Q_m$, where $Q_0=B_1$ and $Q_m=B_2$ are the boundary cycles of $G$,
such that for $0\le i\le m-1$, if $Q_i$ and $Q_{i+1}$ are vertex-disjoint, then every non-contractible $k$-cycle of $G$ drawn in the cylinder between $Q_i$ and $Q_{i+1}$ shares an edge with $Q_i\cup Q_{i+1}$.
Note that we can inherit the flow in $G^\star/F_i^\star$ from $G^\star/F^\star_{i-1}$, and thus in order to find the cycle $Q_{i+1}$, the algorithm
visits only the edges whose dual belongs to $F_{i+1}\setminus F_i$, for $0\le i\le m-1$.  Consequently we can find this sequence of cycles in linear time.

For $0\le i\le m-1$, if $Q_i$ and $Q_{i+1}$ are not vertex-disjoint, then let $S_i$ be the sequence consisting only of $Q_i$.
If $Q_i$ and $Q_{i+1}$ are vertex-disjoint, then let $S_i$ be the sequence obtained by applying algorithm $A_{k+1}$ on
the subgraph of $G$ drawn between $Q_i$ and $Q_{i+1}$ (inclusive) except for its last element $Q_{i+1}$.
We return the concatenation of the sequences $S_0$, $S_1$, \ldots, $S_{m-1}$, and the singleton~$B_2$.

It remains to consider the case that $k=d+1$, and thus every non-contractible cycle in $G$ of length at most $d$ shares an edge
with $B_1$ or $B_2$.  Using the maximum flow algorithm in the dual of $G-E(B_1)$, we find a non-contractible cycle $B'_2$ in $G$ of length at most $d$
closest to and edge-disjoint from $B_1$; this cycle necessarily intersects $B_2$.  If $B'_2$ also intersects $B_1$, we return the sequence $B_1$, $B'_2$, $B_2$.
Hence, suppose that $B'_2$ is disjoint from $B_1$, and let $G_1$ be the subgraph of $G$ between $B_1$ and $B'_2$.
Let $S_2$ be the sequence $B'_2, B_2$ if $B'_2\neq B_2$ and the sequence consisting only of $B_2$ otherwise.
Note that every non-contractible cycle in $G_1$ of length at most $d$ distinct from $B'_2$ shares an edge with $B_1$.

Using the maximum flow algorithm in the dual of $G_1-E(B'_2)$, we find a non-contractible cycle $B'_1$ in $G_1$ of length at most $d$
closest to and edge-disjoint from $B'_2$; this cycle necessarily intersects $B_1$.  If $B'_1$ also intersects $B'_2$,
we return the concatenation of the sequence $B_1, B'_1$ and $S_2$.  Hence, suppose that $B'_1$ is disjoint from $B'_2$,
and let $G_2$ be the subgraph of $G$ between $B'_1$ and $B'_2$.  
Let $S_1$ be the sequence $B_1, B'_1$ if $B'_1\neq B_1$ and the sequence consisting only of $B_1$ otherwise.

By the choice of $B'_1$ and $B'_2$, every non-contractible cycle in $G_2$ of length at most $d$ distinct from $B'_1$ and $B'_2$
shares an edge with both $B'_1$ and $B'_2$.
Finally, for all $e_1\in E(B'_1)$ and $e_2\in E(B'_2)$, we use the maximum flow algorithm in the dual of $G_2-\{e_1,e_2\}$
to determine whether there exists a non-contractible cycle $Q$ of length at most $d$.  If so, we return the
sequence consisting of $S_1$, $Q$, and $S_2$.  If not, then $G_2$ does not contain any non-contractible cycle
of length at most $d$ distinct from $B_1$ and $B_2$, and we return the concatenation of the sequences $S_1$ and $S_2$.
\end{proof}

Using these tools, it is easy to deal with the cylinder case of Theorem~\ref{thm-mainalg}.

\begin{lemma}\label{lemma-algcyl}
Let $d$ be a positive integer. There exists a linear-time algorithm that,
given a quadrangulation $G$ of the cylinder $\Sigma$ with boundary cycles of length at most $d$
and their precoloring $\psi$, decides whether $\psi$ extends to a $3$-coloring of~$G$.
If such a $3$-coloring extending $\psi$ exists, the algorithm outputs one.
\end{lemma}
\begin{proof}
Let $C_1$, \ldots, $C_m$ be the sequence of cycles obtained by applying Lemma~\ref{lemma-splitcyl}.
For $1\le i<j\le m$, let $G_{i,j}$ denote the subgraph of $G$ drawn between $C_i$ and~$C_j$.

For $i=1,\ldots, m$, let $\Psi_i$ denote the set of all $3$-colorings of $C_1\cup C_i$ that extends to a $3$-coloring
of $G_{1,i}$.  We determine the sets $\Psi_1$, \ldots, $\Psi_m$ by dynamic programming,
and test whether $\psi$ belongs to $\Psi_m$.

Clearly, $\Psi_1$ consists of all $3$-colorings of the cycle $C_1$.
Suppose that $i>1$ and that we already determined the set $\Psi_{i-1}$.  To compute $\Psi_i$, it suffices to determine
the set $\Psi'_i$ of all $3$-colorings of $C_{i-1}\cup C_i$ that extend to the subgraph $G_{i-1,i}$.
If the distance between $C_{i-1}$ and $C_i$ is less than $|C_1|+|C_2|$, then let $P$ be a shortest path between $C_{i-1}$ and $C_i$.
By using Observation~\ref{obs-mainalg-disk}, we can determine all $3$-colorings of $C_{i-1}\cup P\cup C_i$ that extend to
a $3$-coloring of $G_{i-1,i}$, and thus also the set $\Psi'_i$.  If the distance between $C_{i-1}$ and $C_i$ is at least $|C_1|+|C_2|$,
then $\Psi'_i$ consist of all $3$-colorings of $C_{i-1}$ and $C_i$ that satisfy the winding number constraint, by Lemma~\ref{lemma-mainalg-cyl1}.

Note that for each element $\psi'$ of $\Psi_i$, we can also keep track of a $3$-coloring of $G_{1,i}$ whose restriction to $C_1\cup C_i$
is equal to $\psi'$.  Hence, we can also return the $3$-coloring of $G$ whose restriction to $B_1\cup B_2$ is equal to $\psi$,
if such a coloring exists.
\end{proof}

The case of a general surface $\Sigma$ is somewhat more involved.
We need to specify how a graph $G$ embedded in $\Sigma$ is represented.  
We use a variant of the polygonal representation; informally, we cut $\Sigma$ along simple non-separating closed
curves and along simple curves between cuffs (without crossing the edges of the drawing of $G$, i.e., passing through
vertices or following the edges of $G$) until we obtain a graph $G$ drawn in a disk $\Delta$.  The surface $\Sigma$ and the graph $G$
can be recovered by gluing parts of the boundary of $\Delta$ back together.

Let us now formally describe the representation; we use a drawing of another graph $H$ to describe the curves along which we cut $\Sigma$.
Let $g$ be the Euler genus of $\Sigma$ and $c$ the number of cuffs of $\Sigma$.
Let $H$ be a graph drawn in $\Sigma$ such that
\begin{itemize}
\item the boundary of every cuff traces a cycle in $H$,
\item for every edge $e\in E(H)$, the curve in $\Sigma$ representing $e$ in the drawing of $H$ is either equal to
a curve representing an edge of $G$, or intersects the drawing of $G$ only in the points representing vertices of $G$,
\item $H$ has exactly one face and this face is homeomorphic to an open disk $\Lambda$, and
\item $H$ has at most $2(g+c-1)$ vertices of degree at least three.
\end{itemize}
The graph $G$ embedded in $\Sigma$ is represented by a graph $G'$ drawn in a closed disk $\Delta$, such that there exists
a continuous surjection $\theta:\Delta\to\Sigma$ satisfying
\begin{itemize}
\item $G'=\theta^{-1}(G)$,
\item the restriction of $\theta$ to the interior of $\Delta$ is a homeomorphism to $\Lambda$, and
\item $\theta$ maps the boundary of $\Delta$ to the boundary walk of the face $\Lambda$ of $H$.
\end{itemize}
Note that the boundary of $\Delta$ contains pairwise internally disjoint closed intervals $A_1$, $B_1$, $A_2$, $B_2$, \ldots, $A_p$, $B_p$ for some $p\le 3(g+c)+1$
such that the restriction of $\theta$ to each of them is injective, $\theta(A_i)=\theta(B_i)$ for $1\le i\le p$,
and $\bigcup_{i=1}^p \theta(A_i)$ is the subgraph of $H$ consisting of edges not contained in any of the boundary cycles of $G$.
Hence, the surface $\Sigma$ is obtained from the disk $\Delta$ by identifying $A_i$ with $B_i$ for $1\le i\le p$, in the
direction prescribed by $\theta$.  We call this representation of an embedded graph a \emph{normal representation}.
Let us remark that a normal representation can be obtained from any other common representation of an embedded graph in linear time.

In~\cite{dvorak2013testing}, we designed a dynamic data structucture for representing first-order properties in sparse graphs.
Here, the following special case will be important.
\begin{lemma}[{Dvo\v{r}\'ak et al.~\cite[Theorem~5.2]{dvorak2013testing}}]\label{lemma-data}
For every $d\ge 0$, there exists a data structure representing a planar graph $G$ and a weight function $E(G)\to \{0,1,\infty\}$, supporting the following operations in constant time
(depending only on $d$):
\begin{itemize}
\item Removal of an edge or an isolated vertex.
\item Changing the weight of any edge.
\item For any vertices $u,v\in V(G)$ and integers $t,w\le d$, deciding whether there exists a path between $u$ and $v$ with
at most $t$ edges and with total weight at most $w$, and finding such a path if it exists.
\end{itemize}
The data structure can be initialized in $O(|V(G)|)$ time.
\end{lemma}

We use it to design a similar data structure for representing embedded graphs (inspired by the algorithm of Cabello and Mohar~\cite{cabello-sepcurv}).

\begin{lemma}\label{lemma-dataemb}
For any integer $d\ge 0$ and every surface $\Sigma$, there exists a data structure as follows.
Let $G_0\cup F$ be a graph embedded in $\Sigma$ such that $F$ is a star forest.
The data structure represents a graph $G$ obtained from $G_0$ by contracting some edges of $F$
and by removing vertices and edges.  The data structure supports the following operations in amortized constant time
(depending only on $d$ and $\Sigma$):
\begin{enumerate}[(a)]
\item Removal of an edge or an isolated vertex.
\item Contraction of an edge of $F$.
\item For any vertex $v\in V(G)$, deciding whether there exists a closed walk $W$ of length at most $d$ with $v\in V(W)$ such that
$W$ is not null-homotopic in $\Sigma$ even after patching a single cuff with a disk, and finding such a walk if that is the case.
\item For any vertex $v\in V(G)$ and any set $D$ of cuffs of $\Sigma$, letting $\hat{\Sigma}$ be the surface obtained from $\Sigma$ by
patching all the cuffs in $D$ and letting $\Lambda\subseteq\hat{\Sigma}$ be an open disk containing all the patches,
deciding whether there exists a closed walk $W$ in $G$ of length at most $d$ such that $W$ contains $v$ and is homotopically
equivalent (in $\Sigma$) to the boundary of $\Lambda$, and finding such a walk if that is the case.
\end{enumerate}
Given a normal representation of $G$, the data structure can be initialized in $O(|V(G)|)$ time.
\end{lemma}
\begin{proof}
Let $G'$ be the graph in a disk $\Delta$ and $\theta:\Delta\to\Sigma$ the surjection that form the normal representation
of $G_0\cup F$, and let $A_1$, $B_1$, \ldots, $A_p$, $B_p$ be the intervals in the boundary of $\Delta$ as in the definition
of the normal representation.
Consider any sequence $s$ of symbols from $\{a_1,b_1,\ldots, a_p,b_p\}$ of length at most $d$ (including the empty one).
We say that $s$ is valid if no appearance of $a_i$ is adjacent to an appearance of $b_i$ in $s$.
For every valid sequence $s$, let $(G^s,\Delta^s,\theta^s)$ be a disjoint copy of $(G',\Delta,\theta)$, and let $A^s_i$ and $B^s_i$ be the corresponding intervals in the boundary of $\Delta^s$.
Glue the copies in a tree-like fashion as follows.  For every non-empty
valid sequence $s=s'a_i$,  identify all points $x\in B^{s'}_i$ and $y\in A^s_i$ such that $\theta^{s'}(x)=\theta^s(y)$.  Similarly, for every non-empty
valid sequence $s=s'b_i$,  identify all points $x\in A^{s'}_i$ and $y\in B^s_i$ such that $\theta^{s'}(x)=\theta^s(y)$.

In this way, we obtained a disk $\Delta_+$ with a graph $G_+$ and a surjection $\theta_+:\Delta_+\to \Sigma$ such that $\theta_+(G_+)$ covers $G_0\cup F$ several times.
Note that for each $v\in V(G_0\cup F)$, the size of $\theta_+^{-1}(v)$ is bounded by a function of $d$, the genus of $\Sigma$ and the number of cuffs of $\Sigma$, which is a constant.
Observe that for every walk $W$ in $G_0\cup F$ of length at most $d$, there exists a walk $W_+$ of the same length in $G_+$ such that $\theta_+(W_+)=W$.
Furthermore, if $W$ is closed, then its homotopy class is uniquely determined by valid sequences $s_1$ and $s_2$ such that the endpoints of $W_+$ lie in $\Delta^{s_1}$ and $\Delta^{s_2}$.
Hence, to find a closed walk $W$ in $G_0\cup F$ of the prescribed homotopy class such that $W$ contains $v$ and has length at most $d$, it is sufficient to try all combinations of $s_1$ and $s_2$ which correspond to this
homotopy class and to test whether there exist vertices in $\Delta^{s_1}\cap \theta_+^{-1}(v)$ and $\Delta^{s_2}\cap \theta_+^{-1}(v)$ joined by a walk of length at most $d$.

To implement the data structure of Lemma~\ref{lemma-dataemb}, we use the data structure of Lemma~\ref{lemma-data} to represent $G_+$ with edge weights set to $1$
for edges corresponding to edges of $G$ and to $\infty$ for those corresponding to edges of $F$.
Contraction of an edge of $F$ is realized by changing the weight of the corresponding (constantly many)
edges to $0$.
Removal of an edge or a vertex from $G$ results in a removal of a constant number of edges or vertices from $G_+$
(if the vertex was obtained by contracting some edges in $F$, we may need to remove more vertices from the data structure,
but this is amortized to the contraction operation).  Hence we can update the data structure in constant time.

Since $F$ is a star forest, the operations (c) and (d) can be implemented as described in the previous paragraph by testing all appropriate homotopy classes in constant time,
using the operation of the data structure of Lemma~\ref{lemma-data} to find paths of length at most $3d+2$ and weight at most $d$.
\end{proof}

The following lemma enables us to restrict our attention to boundary-linked quadrangulations without small essential subgraphs.
We say that a surface $\Sigma'$ is \emph{at most as complex as $\Sigma$} if $\Sigma'$ has smaller genus than $\Sigma$, or
$\Sigma'$ has the same genus and fewer cuffs than $\Sigma$, or $\Sigma'$ is homeomorphic to $\Sigma$.
For a graph $G$ embedded in $\Sigma$, let $b(G)$ denote the multiset of the lengths of the boundary cycles of $G$.
For two multisets $S$, $T$ of integers such that $|S|=|T|=m$, we say that \emph{$S$ dominates $T$} if
there exists an ordering $s_1,\ldots, s_m$ of the elements of $S$ and an ordering $t_1,\ldots,t_m$ of the elements of $T$
such that $s_i\ge t_i$ for $i=1,\ldots, m$.  We say that \emph{$S$ strictly dominates $T$} if $S$ dominates $T$ and $S\neq T$.
Let us also recall the notation $\Sigma_h$ and $G_h$ was defined before Theorem~\ref{thm-struct}.

\begin{lemma}\label{lemma-qstruc}
Let $\nu(\Sigma,k)$ be any function.  For any surface $\Sigma$ and an integer $k\ge 0$, there exists a constant $\sigma$ and a linear-time algorithm as follows.
Let $G$ be a graph $2$-cell embedded in $\Sigma$ with boundary cycles $B_1$, \ldots, $B_c$ of total length at most $k$.
The algorithm returns a subgraph $H$ of $G$ with at most $\sigma$ vertices such that $B_1\cup \ldots\cup B_c\subseteq H$
and for each face $h$ of $H$, $G_h$ (in its embedding in $\Sigma_h$) does not contain any
connected essential subgraph with fewer than $\nu(\Sigma_h,k_h)$ edges, where $k_h$ is the sum of the lengths
of the boundary cycles of $G_h$.  In addition, if $\Sigma_h$ is not the cylinder, then $G_h$ is boundary-linked.
Furthermore, $\Sigma_h$ is at most as complex as $\Sigma$, and if $\Sigma_h$ is homeomorphic to $\Sigma$, then
$b(G)$ dominates $b(G_h)$.
\end{lemma}
\begin{proof}
We proceed by induction on the complexity of the surface and on $k$, i.e., we assume that
the algorithm exists for surfaces at most as complex as $\Sigma$ that either
are not homeomorphic to $\Sigma$, or the total length of their boundary cycles is less than $k$.

The claim is obvious if $\Sigma$ is the sphere (we can take $H$ to be empty).
If $\Sigma$ is the cylinder, we test whether the distance between $B_1$ and $B_2$ in $G$ is at least $\nu(\Sigma, |B_1|+|B_2|)$.
If so, then $G$ does not contain any connected essential subgraph with fewer than $\nu(\Sigma, |B_1|+|B_2|)$ edges by Lemma~\ref{lemma-schism},
and we return $H=B_1\cup B_2$.  Otherwise, let $P$ be the shortest path between $B_1$ and $B_2$, and let $f$ be the unique face of
$B_1\cup B_2\cup P$.  We apply the algorithm from the induction hypothesis to the graph $G_f$ embedded in a disk, and return the graph obtained
from the result by identifying the two paths in its boundary corresponding to~$P$.

Suppose that $\Sigma$ is the disk.  If $G$ is boundary-linked, we can take $H=B_1$.  Otherwise, by performing breadth-first search
from every vertex of $B_1$, we can in linear time find
a spoke $P$ of $B_1$ such that both bases of $P$ have length less than $|P|$.  Let $f_1$ and $f_2$ be the faces of $B_1\cup P$.
Note that both $G_{f_1}$ and $G_{f_2}$ are embedded in disks with boundary cycles of length less than $k$, and thus by the induction hypothesis,
we can find their subgraphs $H_1$ and $H_2$, respectively, satisfying the conclusions of Lemma~\ref{lemma-qstruc}.
We set $H=H_1\cup H_2$.

Hence, assume that $\Sigma$ either has non-zero genus or at least three cuffs.
We build the data structure of Lemma~\ref{lemma-dataemb} (with $F$ empty), and by querying all vertices if necessary,
we either find a closed walk $W$ satisfying one of the following conditions or decide that there is no such walk:
\begin{enumerate}[(i)]
\item $W$ has length at most $2\nu(\Sigma, k)+2k$ and it is not null-homotopic in $\Sigma$ even after patching a single cuff with a disk, or
\item $W$ is a cycle homotopically equivalent to a boundary cycle $B_i$ for some $i\in \{1,\ldots, c\}$ and $|W|<|B_i|$.
\end{enumerate}
Using Lemma~\ref{lemma-schism}, it is easy to see that that if no such walk exists, then $G$ is boundary-linked and contains no essential subgraph with
less than $\nu(\Sigma,k)$ edges; hence, we can set $H=B_1\cup \ldots\cup B_c$.

Suppose that a walk $W$ as in (i) exists.  Then we can find a connected subgraph $S$ of $G$ with $E(S)\subseteq E(W)$ such that
$S$ satisfies one of the conclusions of Lemma~\ref{lemma-schism}.  For each face $f$ of $S'=S\cup B_1\cup \ldots \cup B_c$, the surface $\Sigma_f$ has
either smaller genus than $\Sigma$, or the same genus and fewer cuffs.  Similarly, if a walk $W$ as in (ii) exists, then
each face $f$ of $S'=W\cup B_1\cup \ldots \cup B_c$ is either an open cylinder, or $\Sigma_f$ is homeomorphic to $\Sigma$ and
$b(G)$ strictly dominates $b(G_f)$.

Hence, in both cases we can apply the induction hypothesis for $G_f$, obtaining its subgraph $H_f$.  We let $H$ be the union of the graphs $H_f$ over all faces $f$ of~$S'$.
\end{proof}

To deal with the boundary-linked case without small essential subgraphs, we use the following shrinking lemma.
Let $G$ and $F$ be graphs embedded in the same surface such that the embeddings only intersect in vertices.  A \emph{diagonal} of a $4$-face $f=v_1v_2v_3v_4$ of $G$ is
an edge $e$ of $F$ joining either $v_1$ with $v_3$, or $v_2$ with $v_4$, and drawn inside $f$.  By \emph{contracting the diagonal} $e=v_iv_{i+2}$, we mean
identifying the ends of $e$ to a single vertex $v$ (modifying the embeddings of $G$ and $F$ in the natural way) and suppressing the arising $2$-faces $v_{i+1}v$ and $v_{i-1}v$ (where $v_0=v_4$).

\begin{lemma}\label{lemma-shrink}
For any surface $\Sigma$ other than the sphere with at most two cuffs, and for all integers $k,\nu\ge 0$, there exists a linear-time algorithm as follows.
Let $G$ be a boundary-linked quadrangulation of $\Sigma$ with boundary cycles $B_1$, \ldots, $B_c$ of total length at most $k$,
such that $G$ does not contain any connected essential subgraph with fewer than $\nu$ edges.
Let $F$ be a star forest embedded in $\Sigma$ so that the embeddings of $F$ and $G$ intersect only in vertices,
such that each edge of $F$ is a diagonal of a $4$-face of $G$.

The algorithm returns a boundary-linked quadrangulation $G'$ of $\Sigma$ and a forest $F'$ embedded in $\Sigma$ obtained from $G$ and $F$, respectively, by contracting some of the diagonals of $F$,
such that $G'$ does not contain any connected essential subgraph with fewer than $\nu$ edges,
and a surface $\Sigma'\subseteq \Sigma$ homeomorphic to $\Sigma$ such that $G'\cap \Sigma'$ is a quadrangulation
of $\Sigma'$ with boundary cycles $B'_1$, \ldots, $B'_c$ of the same length as the corresponding boundary cycles of $G$,
and one of the following claims holds.
\begin{itemize}
\item $G'\cap \Sigma'$ contains a connected essential subgraph with at most $2\nu+2k+2$ edges, or
\item $F'\cap \Sigma'=\emptyset$.
\end{itemize}
\end{lemma}
\begin{proof}
We use the data structure of Lemma~\ref{lemma-dataemb}.  Initially, set $\Sigma'=\Sigma$, $G'=G$ and $F'=F$.  We process the diagonals of $F$ one by one.
If the currently processed diagonal $uv$ does not lie in $\Sigma'$, we ignore it.  Otherwise, we contract $uv$ in $G'$ to a vertex $w$, obtaining a new graph $G''$ with forest $F''$.
We use the data structure to test whether $G''\cap \Sigma'$ contains a closed walk $W'$ containing $w$,
such that either $W'$ has length at most $2\nu+2k$ and it is not null-homotopic in $\Sigma$ even after patching a single cuff with a disk,
or $W'$ is a cycle homotopically equivalent to a bountary cycle $B_i$ for some $i\in \{1,\ldots, c\}$ and $|W'|<|B_i|$.
If no such closed walk exists, observe that $G''\cap \Sigma'$ is boundary linked and does not contain any connected essential subgraph with fewer than $\nu$ edges.
In this case, we set $G'=G''$, $F'=F''$ and proceed with further diagonals of $F$.

If $W'$ exists and is not null-homotopic in $\Sigma$ even after patching a single cuff with a disk, then we set $G'=G''$, $F'=F''$, and the algorithm ends---the graph $G''\cap \Sigma'$ contains a connected essential subgraph with at most $2\nu+2k+2$ edges.

Suppose now that $W'$ exists and is homotopically equivalent to a boundary cycle $B_i$ and $|W'|<|B_i|$.  Since $G''$ is a quadrangulation, $|W|$ and $|B'_i|$ have the same parity.
And since $G'$ is boundary-linked, we conclude that $|W'|=|B_i|-2$. Note that $G'$ contains a cycle $W\neq B_i$ of length $|B_i|$ that is homotopically equivalent to $B_i$,
such that the diagonal $uv$ is contained in the part of $\Sigma'$ between $B_i$ and $W$.
In this case, we alter $\Sigma'$ by removing the cylinder between $B_i$ and $W$, excluding the cycle $W$ which becomes a new boundary cycle;
to reflect this in the data structure, it suffices to remove all the vertices and edges of $G'$ that belong to the removed cylinder.
We keep the current graph $G'$ and forest $F'$ and we proceed with processing futher diagonals of $F$.

Note that in the previous paragraph, the removed cylinder contains the diagonal $uv$.  Hence, if the algorithm processes all diagonals of $F$, then $G'\cap \Sigma'$ contains no non-contracted diagonals of $F$
as required by the second outcome of the lemma.
\end{proof}

Next, we design an algorithm to get rid of short contractible separating cycles.

\begin{lemma}\label{lemma-elim24}
For any surface $\Sigma$, there exists a linear-time algorithm that for any graph $G$ with a $2$-cell embedding in $\Sigma$
returns its subgraph $H$ such that
\begin{itemize}
\item all boundary cycles of $G$ belong to $H$,
\item all contractible cycles in $H$ of length at most $4$ bound $2$-cell faces, and
\item all vertices and edges of $G$ that do not belong to $H$ are drawn in $2$-cell $(\le\!4)$-faces of $H$.
\end{itemize}
\end{lemma}
\begin{proof}
First, let us prepare a data structure similarly to Lemma~\ref{lemma-dataemb}.
Let $G'$ be the graph drawn in a disk $\Delta$ and $\theta:\Delta\to\Sigma$ the surjection that form the normal representation
of $G$, and let $A_1$, $B_1$, \ldots, $A_p$, $B_p$ be the intervals in the boundary of $\Delta$ as in the definition of the normal representation.
For every valid sequence $s$ of length at most two, let $(G^s,\Delta^s,\theta^s)$ be a disjoint copy of $(G',\Delta,\theta)$, and let $A^s_i$ and $B^s_i$ be the corresponding intervals in the boundary of $\Delta^s$.
Glue the copies in a tree-like fashion as in the proof of Lemma~\ref{lemma-dataemb}.

In this way, we obtained a graph $G_0$ drawn in a disk $\Delta_0$ and a surjection $\theta_0:\Delta_0\to \Sigma$ such that
$G$ is covered $(4p^2+1)$ times by $\theta_0(G_0)$.  For every closed walk $W_0$ in $G_0$, the walk $\theta_0(W_0)$ in $G$ is contractible.
Conversely, for every contractible closed walk $W$ in $G$ of length at most $4$, there exists a closed walk $W_0$ of length $4$ in $G_0$
such that $W=\theta_0(W_0)$.  Build the data structure of Lemma~\ref{lemma-data} for $G_0$ with $d=3$, initally setting the weights of
all edges to $1$.

Now, given an edge $e=uv$ of $G$, we can determine whether $e$ is contained in a non-facial contractible $(\le\!4)$-cycle in $G$
in constant time as follows.  We say a set $X\subset E(G)$ \emph{breaks $(\le\!4)$-faces} at $e$ if $X$ contains exactly one edge of
each $(\le\!4)$-face incident with $e$ and no other edges, and $e\not\in X$.  Note that there are at most $9$ sets that break
$(\le\!4)$-faces at $e$, and given the embedding of $G$, they can be enumerated in constant time.
Let $e_0=u_0v_0$ be an edge of $G_0$ such that $\theta_0(e_0)=e$.
For each set $X$ that breaks $(\le\!4)$-faces at $e$, set the weights of edges in $\theta_0^{-1}(X\cup \{e\})$ to $\infty$, test
whether $G_0$ contains a walk of length and weight at most $3$ from $u_0$ to $v_0$, and restore the weights of all edges to $1$.
If such a walk is found, then its image together with the edge $uv$ forms a non-facial contractible $(\le\!4)$-cycle in $G$.
Otherwise, no such cycle exists.

Now, for each edge $e\in E(G)$, we try to find such a cycle $C$, and if it exists, we remove all the vertices and edges of $G$ contained
in the open disk bounded by $C$ (and update the data structure accordingly).  We repeat this procedure for each edge $e$ until
all contractible cycles of length at most $4$ that contain $e$ bound $2$-cell faces.  Note that we remove each vertex and edge at most
once, hence the total time complexity of the algorithm is linear.
\end{proof}

Finally, we need the following result.
\begin{theorem}[Ne\v{s}et\v{r}il and Ossona de Mendez~\cite{grad2}]\label{thm-star}
For every $g\ge 0$, there exists $\mu\ge 0$ and a linear-time algorithm
that for a simple graph of Euler genus at most $g$ finds a coloring of its vertices
by at most $\mu$ colors such that the union of every two color classes induces a star forest.
\end{theorem}

The algorithm of Theorem~\ref{thm-mainalg} is now straightforward.

\begin{proof}[Proof of Theorem~\ref{thm-mainalg}]
We proceed by induction on the complexity of the surface, i.e., we assume that
the algorithm exists for quadrangulations of surfaces at most as complex as $\Sigma$ that
are not homeomorphic to $\Sigma$.
By Observation~\ref{obs-mainalg-disk} and Lemma~\ref{lemma-algcyl}, we can assume that $\Sigma$
either has non-zero genus or at least three cuffs.

Let $\nu(\Sigma,k)\ge 5$ be a function defined so that for each surface $\Sigma$ and integer $k\ge 0$,
Theorem~\ref{thm-main} holds with $\nu\colonequals \nu(\Sigma,k)$. Let us apply the algorithm of Lemma~\ref{lemma-qstruc} with this function $\nu$; let $H$ be the resulting subgraph of $G$.
For each extension $\psi'$ of $\psi$ to $H$, we will test whether $\psi'$ extends to a $3$-coloring of $G_h$
for every face $h$ of $H$.  If that is the case, we also find the extensions, which together give a $3$-coloring of $G$ extending $\psi$.
Hence, it suffices to discuss how to extend $\psi'$ to $G_h$ for a face $h$ of $G$.

If $\Sigma_h$ is not homeomorphic to $\Sigma$, we apply the algorithm which exists by the induction hypothesis.  Hence, assume that $\Sigma_h$ is
homeomorphic to $\Sigma$, and thus $\Sigma_h$ is not a cylinder and the sum of lengths of boundary cycles of $G_h$ is at most~$k$.
Furthermore, $G_h$ is boundary-linked and contains no essential subgraph with less than $\nu(\Sigma,k)$ edges.
By Theorem~\ref{thm-main}, $\psi'$ extends to $G_h$ if and only if $\psi'$ satisfies the winding number constraint,
and this can be tested in linear time.

Suppose that $\psi'$ satisfies the winding number constraint, and thus we need to find a $3$-coloring of $G_h$ that extends $\psi'$.
We construct a sequence $(G_0,\Sigma_0,\psi_0)$, $(G_1,\Sigma_1,\psi_1)$, \ldots, $(G_r,\Sigma_r,\psi_r)$,
where for $0\le i\le r$, $G_i$ is a boundary-linked quadrangulation of $\Sigma_i$ with no essential subgraph with less than $\nu(\Sigma,k)$ edges,
$\Sigma_i$ is homeomorphic to $\Sigma$ and $\psi_i$ is a $3$-coloring of the boundary cycles of $G_i$, as follows.
Let $n_i$ denote the number of vertices of $G_i$ that are not contained in
its boundary cycles.  

Let $G_0=G_h$, $\Sigma_0=\Sigma_h$ and let $\psi_0$ be the restriction of $\psi'$ to the boundary cycles of $G_h$.
Suppose that we already constructed $G_i$ for some $i\ge 0$.
To obtain $G_{i+1}$, first take its subgraph using Lemma~\ref{lemma-elim24} and suppress the resulting $2$-faces. 
Let $G'_i$ denote the resulting graph.  Since $G_i$ is boundary-linked, observe that $G'_i$ contains no parallel edges and that $G$ does not contain
three distinct $4$-faces $u_1v_1w_1x_1$, $u_2v_2w_2x_2$ and $u_3v_3w_3x_3$ such that $u_1=u_2=u_3$ and $w_1=w_2=w_3$.  Let $n'_i$ denote the number of vertices of $G'_i$ that are not contained in
its boundary cycles.

Next, let $F'_i$ be a maximal graph embedded in $\Sigma_i$ such that each edge of $F'_i$ is a diagonal of a face of $G'_i$ and $F'_i$ has no parallel edges.
Note that $|E(F'_i)|$ is at least half the number of faces
of $G_i$, and thus $|E(F'_i)|\ge n'_i/2$.  By Theorem~\ref{thm-star}, there exists a subgraph $F_i$ of $F'_i$ with at least $\frac{n'_i}{\mu^2}$ edges such that $F_i$ is a star forest.
Apply Lemma~\ref{lemma-shrink} to $G'_i$ and $F_i$, let $G''_i$ and $\Sigma_{i+1}$ be the resulting graph and surface, and let $G_{i+1}=G''_i\cap \Sigma_{i+1}$.
Use Lemma~\ref{lemma-algcyl} to find a $3$-coloring $\psi'_i$ of $G''_i\cap\overline{\Sigma_i\setminus\Sigma_{i+1}}$ that extends $\psi_i$, and let $\psi_{i+1}$
be the restriction of this coloring to the boundary cycles of $G_{i+1}$.
If $G_{i+1}$ contains a connected essential subgraph with at most $2\nu(\Sigma,k)+2k+2$ edges, then let $r\colonequals i+1$, otherwise proceed with the construction.

Note that if $i<r-1$, then $G_{i+1}$ contains no non-contracted diagonals of $F_i$, and thus $n_{i+1}\le n'_i-|F_i|\le (1-1/\mu^2)n'_i\le (1-1/\mu^2)n_i$.
Observe that $|V(G_i)|\le 3n_i$, since $G_i$ is boundary-linked and does not contain an essential subgraph with at most two edges.
Hence, each step of the construction has time complexity linear in $n_i$, and thus the total time complexity for the construction of the sequence
is $O\left(\sum_{i=0}^{r-1} n_i\right)\le O\left(|V(G)|\sum_{i\ge 0} (1-1/\mu^2)^i\right)=O(|V(G)|)$.

Now, $G_r$ contains a connected essential subgraph with at most $2\nu(\Sigma,k)+2k+2$ edges.  Let $H_r$ be the union of the boundary cycles of $G_r$ with this essential subgraph,
and observe that for each face $h$ of $H_r$, the surface $\Sigma_{r,h}$ is at most as complex as $\Sigma$ and not homeomorphic to $\Sigma$.  Hence, we can find
a $3$-coloring of $G_r$ extending $\psi_r$ (which exists by Theorem~\ref{thm-main}) by trying all the possible extensions of $\psi_r$ to $H_r$ and for each of them
applying the induction hypothesis to the subgraphs drawn in the faces of $H_r$.

Furthermore, for $0\le i\le r-1$, given a $3$-coloring $\varphi_{i+1}$ of $G_{i+1}$ that extends $\psi_{i+1}$, we can obtain a coloring of $G_i$ that extends $\psi_i$ in linear time as follows.
First, let $\varphi''_i=\psi'_i\cup \varphi_{i+1}$ be a $3$-coloring of $G''_i$ that extends $\psi_i$.   Next, to all vertices that were identified to a single vertex $w$ by
contracting the diagonals in $F_i$ give the same color as $w$, thus obtaining a $3$-coloring $\varphi'_i$ of $G'_i$ that extends $\psi_i$.  Finally,
use the algorithm of Observation~\ref{obs-mainalg-disk} to extend this coloring to the parts of $G_i$ removed during the construction of $G'_i$.

Then, $\varphi_0$ is a $3$-coloring of $G_h$ which extends $\psi'$, as required.
\end{proof}

\section{Important subgraphs of embedded graphs}\label{sec-struct}

We need a stronger form of Theorem~\ref{thm:ctvrty} which deals with graphs with precolored cycles.  First, let us give several definitions.
Let $B$ be a subgraph of $G$.  We say that $G$ is \emph{$B$-critical} if $G\neq B$ and for every proper subgraph $G'$ of $G$
such that $B\subseteq G'$, there exists a $3$-coloring of $B$ that extends to a $3$-coloring of $G'$, but not to a $3$-coloring of $G$.
Note that $G$ is $4$-critical if and only if $G$ is $\emptyset$-critical.  

Suppose that a graph $G$ is embedded in a surface $\Sigma$ so that every cuff of $\Sigma$ traces a cycle in $G$.  To each face $f$ of $G$,
we assign a weight $w_0(f)$ as follows.  If $f$ is homeomorphic to an open disk, then let
$$w_0(f)=\begin{cases}
0&\text{if } |f|\le 4\\
4/4113&\text{if } |f|=5\\
72/4113&\text{if } |f|=6\\
540/4113&\text{if } |f|=7\\
2184/4113&\text{if } |f|=8\\
|f|-8&\text{if } |f|\ge 9.
\end{cases}$$
If $f$ is not homeomorphic to an open disk, then let $w_0(f)=|f|$.
For a surface $\Pi$ of Euler genus $g$ with $c$ cuffs, let $s(\Pi)=6c-6$ if $g=0$ and $c\le 2$, and $s(\Pi)=120g+48c-120$ otherwise.
Recall that $\Sigma_f$ was defined prior to Theorem~\ref{thm-struct}.
For a real number $\eta$ and a face $f$ of $G$, let $w_\eta(f)=w_0(f)+\eta s(\Sigma_f)$.
Let $$w_\eta(G)=\sum_{\text{$f$ face of $G$}} w_\eta(f).$$

\begin{theorem}[{Dvo\v{r}\'ak et al.~\cite[Corollary~5.5]{trfree4}}]\label{thm-ctvrtygen}
There exists a constant $\eta$ such that the following holds.
Let $G$ be a triangle-free graph embedded in a surface $\Sigma$ without non-contractible $4$-cycles, so that every cuff of $\Sigma$ traces a cycle in $G$,
and let $B$ be the union of boundary cycles of $G$.  If $G$ is $B$-critical, then $w_\eta(G)\le w_\eta(B)$.
\end{theorem}

By iterating Theorem~\ref{thm-ctvrtygen}, we obtain a variant of Theorem~\ref{thm-struct}.
\begin{lemma}\label{lemma-almstruct}
Let $\eta$ be the constant of Theorem~\ref{thm-ctvrtygen}.  Let $G$ be a triangle-free graph embedded in a surface $\Sigma$ without non-contractible $4$-cycles, so that every cuff of $\Sigma$ traces a cycle in $G$, and
let $B$ be the union of boundary cycles of $G$.  There exists a subgraph $H$ of $G$ such that $B\subseteq H$,
$w_\eta(H)\le w_\eta(B)$ and for every face $h$ of $H$, every $3$-coloring of the boundary of $h$ extends to a $3$-coloring of $G_h$.
\end{lemma}
\begin{proof}
Let $H$ be a maximal subgraph of $G$ such that $B\subseteq H$ and $w_\eta(H)\le w_\eta(B)$.  Consider any face $h$ of $H$, and let
$B_h$ be the union of boundary cycles of $G_h$ in its embedding in $\Sigma_h$.  If there exists a $3$-coloring of $B_h$ that does not extend to a $3$-coloring of $G_h$,
then $G_h$ contains a $B_h$-critical subgraph $H_h$.  Let $H'_h$ be the subgraph of $G$ corresponding to $H_h$.
By Theorem~\ref{thm-ctvrtygen}, we have $w_\eta(H_h)\le w_\eta(B_h)$, and thus $w_\eta(H\cup H'_h)\le w_\eta(H)\le w_\eta(B)$.  This contradicts the maximality of $H$.
Therefore, $H$ satisfies the conclusions of Lemma~\ref{lemma-almstruct}.
\end{proof}

Note that Lemma~\ref{lemma-almstruct} almost implies Theorem~\ref{thm-struct}, up to non-contractible $4$-cycles and quadrangulations.
The former are easy to deal with by cutting the surface, and we already analyzed quadrangulations in Lemma~\ref{lemma-qstruc}.

\begin{proof}[Proof of Theorem~\ref{thm-struct}.]
We proceed by induction on the complexity of the surface and on $k$, i.e., we assume that
the theorem holds for graphs embedded in surfaces at most as complex as $\Sigma$ that either
are not homeomorphic to $\Sigma$, or the total length of their boundary cycles is less than $k$.

Suppose first that
\begin{itemize}
\item $G$ contains a connected essential subgraph $K$ with at most $8$ edges, or 
\item $\Sigma$ is not the cylinder and $G$ contains a boundary cycle $B$ and a cycle $K$ surrounding the incident cuff such that $|K|<|B|$, or
\item $\Sigma$ is the cylinder and $G$ contains a non-contractible cycle $K$ shorter than both boundary cycles of $G$.
\end{itemize}
Then, we cut the surface along $K$, apply the induction hypothesis to the pieces and let $H$ be the union of the resulting subgraphs.  Hence, assume that $G$ has no such subgraph.

In particular, every non-contractible $4$-cycle surrounds a cuff bounded by a $4$-cycle.  If $\Sigma$ is a cylinder and both of its boundary cycles $B_1$ and $B_2$ have length $4$, then we let $H=B_1\cup B_2$.
The graph $H$ clearly satisfies the conclusion of the theorem.
Otherwise, let $B_1$, \ldots, $B_t$ be the boundary cycles of $G$ of length $4$, and for $1\le i\le t$, let $K_i$ be a $4$-cycle surrounding the cuff of $B_i$ such that the subset $\Sigma_i$ of $\Sigma$ between $B_i$ and $K_i$ is maximal.
Note that for $i\neq j$, we have $\Sigma_i\cap \Sigma_j=\emptyset$, since $G$ does not contain an essential subgraph with at most $8$ edges and $\Sigma$ is not a cylinder with two boundary $4$-cycles.
Without loss of generality, we can assume that $B_i$ and $K_i$ are vertex-disjoint for $1\le i\le t'$ and share a vertex $v_i$ for $t'+1\le i\le t$.  For $1\le i\le t$, let $v_i$ be an arbitary vertex of $K_i$.
Let $\Sigma'=\overline{\Sigma\setminus \bigcup_{i=1}^{t'}\Sigma_i}$ and $G'=G\cap \Sigma'$.  Note that $\Sigma'$ is homeomorphic to $\Sigma$.  Let $G''$ be obtained from $G'$ by, for $1\le i\le t$, splitting the vertex
$v_i$ into two vertices $v'_i$ and $v''_i$ and adding a new vertex $v'''_i$ adjacent to $v'_i$ and $v''_i$ (drawn in the boundary of $\Sigma'$), so that $G''$ contains no non-contractible $4$-cycles.

Let $H_1$ be the subgraph of $G''$ obtained using Lemma~\ref{lemma-almstruct}, and let $H_2$ be the subgraph of $H_1$ consisting of the boundary cycles of $G''$ and of all boundary walks of faces of $H_1$ of length greater than $4$.
Since $w_\eta(H_1)\le 5/4k+\eta s(\Sigma)$ and since each face of length greater than $4$ is incident with at most $4113w_\eta(f)$ edges, it follows that
$H_2$ has at most $5/4k+4113(5/4k+\eta s(\Sigma))$ edges.  Furthermore, for each face $h$ of $H_2$, either $h$ is a face of $H_1$, and thus every precoloring of its boundary extends to a $3$-coloring of $G''_h$,
or all faces of $H_1$ contained in $H$ have length $4$.  Since $G''$ does not contain non-contractible $4$-cycles, and since every contractible $4$-cycle bounds a face of $G$ by the assumptions of Theorem~\ref{thm-struct},
it follows that $G''_h$ is a quadrangulation in this case.

Let $\nu(\Sigma,k)\ge 5$ be function defined so that for each surface $\Sigma$ and integer $k\ge 0$,
Theorem~\ref{thm-main} holds with $\nu\colonequals \nu(\Sigma,k)$.
Let $H_3$ be obtained from $H_2$ as follows: for each face $h$ of $H_2$ such that $G''_h$ is a quadrangulation,
we apply Lemma~\ref{lemma-qstruc} to $G''_h$ and replace the face $h$ by the resulting subgraph.
Hence, the size of $H_3$ is bounded by some constant depending only on $\Sigma$ and $k$, and every face $h$ of $H_3$ satisfies one of the following:
\begin{itemize}
\item $h$ is a face of $H_1$, and thus $h$ satisfies (a) of Theorem~\ref{thm-struct}, or
\item $G_h$ is a quadrangulation and $h$ is an open cylinder, and thus $h$ satisfies (c) of Theorem~\ref{thm-struct}, or
\item $G_h$ is a boundary-linked quadrangulation with no connected essential subgraph with fewer than $\nu(\Sigma_h,k_h)$ edges, where $k_h$ is the sum of the lengths of the boundary cycles of $G_h$.
In this case Theorem~\ref{thm-main} implies that $h$ satisfies (b) of Theorem~\ref{thm-struct}.
\end{itemize}

Finally, let $H$ be obtained from $H_3$ by contracting edges $v'_iv'''_i$ and $v'_iv'''_i$ and adding the cycle $B_i$ for $1\le i\le t'$.  Note that the face of $G''$ incident with $v'_iv'''_iv''_i$
does not have length $4$, and thus the face of $H_3$ incident with this path satisfies (a) of Theorem~\ref{thm-struct}; this is not changed by eliminating the vertex $v'''_i$ of degree two.
Furthermore, the faces of $H$ incident with $B_1$, \ldots, $B_{t'}$ satisfy (d) of Theorem~\ref{thm-struct}.
\end{proof}

Before we prove Corollary~\ref{cor-ew}, we need some results on coloring planar graphs.

\begin{theorem}[Gimbel and Thomassen~\cite{gimbel}]\label{thm-disk}
Let $G$ be a triangle-free graph embededded in a disk $\Delta$ with a boundary cycle $B$ of length at most $6$.
If some precoloring of $B$ does not extend to a $3$-coloring of $G$, then $|B|=6$ and
$G$ has a subgraph that quadrangulates $\Delta$.
\end{theorem}

As a corollary, we obtain the following.

\begin{lemma}\label{lemma-colfloat}
Let $G$ be a connected triangle-free plane graph and let $B$ be the boundary walk of its outer face.
Then $G$ has a $3$-coloring $\psi_1$ with the winding number satisfying $|\omega_{\psi_1}(B)|\le 1$.  Furthermore, if
$G$ does not contain a subgraph $Q$ such that $B\subseteq Q$ and all inner faces of $Q$ have length $4$,
then $G$ has a $3$-coloring $\psi_2$ with $1\le |\omega_{\psi_2}(B)|\le 2$.
\end{lemma}
\begin{proof}
Let $B=b_1b_2\ldots b_k$.  Let $G'$ be the graph obtained from $G$ by adding two cycles $v'_1v'_2\ldots v'_k$ and $v_1v_2\ldots v_k$
and edges $b_iv'_i$ and $v'_iv_i$ ($1\le i\le k$) and $v_1v_{2+2i}$ ($1\le i\le \lfloor k/2\rfloor - 2$).
Let $B'=v_{2\lfloor k/2\rfloor - 2}v_{2\lfloor k/2\rfloor - 1}\ldots v_1$.
Note that $4\le |B'|\le 5$.  By Gr\"otzsch's Theorem, $G'$ has a $3$-coloring $\psi$.  Note that
$|\omega_\psi(B')|\le 1$ and by Corollary~\ref{cor-win0}, $\omega_\psi(B)=\omega_\psi(B')$.
Hence, we can choose $\psi_1$ as the restriction of $\psi$ to $G$.

Now, suppose that $G$ does not contain a subgraph $Q$ as described in the statement of the lemma,
and in particular $k\ge 5$.  If $k$ is odd, then $|\omega_\psi(B')|=1$ and the claim of Lemma~\ref{lemma-colfloat} follows
by setting $\psi_2=\psi_1$.
Suppose that $k$ is even.  Let $G''=G'-v_1v_{k-2}$ and $B''=v_{k-4}v_{k-3}\ldots v_1$.
We have $|B''|=6$ and $G''$ does not contain a quadrangulation, and thus by Theorem~\ref{thm-disk},
it has a $3$-coloring $\psi$ such that $\psi(v_{k-4})=\psi(v_{k-1})=1$, $\psi(v_{k-3})=\psi(v_k)=2$ and
$\psi(v_{k-2})=\psi(v_1)=3$.  Note that $\omega_\psi(B)=\omega_\psi(B'')=2$, and thus we can set $\psi_2$ to be the restriction
of $\psi$ to $G$.
\end{proof}

Now, we are ready to deal with the graphs of large edge-width.

\begin{proof}[Proof of Corollary~\ref{cor-ew}.]
By Gr\"otzsch's theorem, we can assume that $\Sigma$ is not the sphere.
Let $\gamma=\max(5,\beta+1)$, where $\beta$ is the constant of Theorem~\ref{thm-struct}.  We can assume that $G$ is $4$-critical,
and thus by Theorem~\ref{thm-disk}, every $4$-cycle in $G$ bounds a face.  Let $H$ be the subgraph of $G$ as in Theorem~\ref{thm-struct}.
Since every cycle in $H$ has length less than $\gamma$, it follows that all cycles in $H$ are contractible.  Therefore, $H$ has exactly
one face $h$ such that $\Sigma_h$ has non-zero genus.  Let $G_0$ be the subgraph of $G$ drawn in the closure of $h$, and let $G_1$, \ldots,
$G_m$, be the subgraphs of $G$ drawn in the components of $\Sigma\setminus h$.  Observe that for $1\le i\le m$, $G_i$ is a plane graph
with the outer face bounded by a closed walk $B_i$, such that $B_i$ is one of the facial walks of $h$.

The face $h$ satisfies (a) or (b) of Theorem~\ref{thm-struct}.  If it satisfies (a), then we can $3$-color each of the graphs $G_1$, \ldots, $G_k$ arbitrarily
using Gr\"otzsch's theorem and extend the coloring to $G_0$.  Hence suppose that $h$ satisfies (b), i.e., $G_h$ is a quadrangulation
and any precoloring of $B_1\cup \ldots\cup B_m$ which satisfies the winding number constraint extends to a $3$-coloring of $G_0$.
Since the dual of $G_h$ has even number of vertices of odd degree, it follows that even number of the boundary walks $B_1$, \ldots, $B_m$ has odd length.

If $\Sigma$ is orientable, then $3$-color $G_1$, \ldots, $G_m$ using Lemma~\ref{lemma-colfloat} so that the winding number of each of the boundary walks
is $0$ or $\pm 1$.  Note that by permuting colors, a coloring with winding number $1$ on some walk $W$ can be transformed to a coloring with winding number $-1$ on $W$.
Therefore, we can choose the colorings so that the sum of winding numbers of $B_1$, \ldots, $B_m$ is $0$, and thus it satisfies the winding number constraint in $G_0$.
Thus, we can extend it to a $3$-coloring of $G$.

Suppose now that $\Sigma$ is non-orientable.  If $G$ is a quadrangulation, then the assumptions of Corollary~\ref{cor-ew} imply that it satisfies the winding number
constraint, and thus $G$ is $3$-colorable by Theorem~\ref{thm-main}.  Hence, suppose that say $G_1$ has an inner face of length greater than $4$.
Again, choose the coloring $\psi$ of $G_1\cup \ldots\cup G_m$ so that the winding number of each of the boundary walks is $0$ or $\pm 1$ and the sum of the winding numbers
of boundary walks is $0$.  If this coloring is parity-compliant in $G_0$, then we can extend it to a $3$-coloring of $G$.  Otherwise, we alter the coloring of $G_1$.
If $|\omega_\psi(B_1)|=1$, say $\omega_\psi(B_1)=-1$, we permute the colors in the coloring of $G_1$ so that $\omega_\psi(B_1)=1$.
If $\omega_\psi(B_1)=0$, we use Lemma~\ref{lemma-colfloat} to find a coloring of $G_1$ so that $\omega_\psi(B_1)=2$.
In both cases, we obtain a $3$-coloring of $G_1\cup \ldots\cup G_m$ such that the sum of winding numbers of boundary cycles is $2$, and thus it is parity-compliant in $G_0$
and can be extended to a $3$-coloring of $G$.
\end{proof}

\bibliographystyle{acm}
\bibliography{sesty}

\end{document}